\documentclass[12pt]{amsart}
\usepackage[utf8]{inputenc}
\usepackage{amsfonts}
\usepackage{amssymb}
\usepackage{amsmath}
\usepackage[english]{babel}
\usepackage{amsthm}
\theoremstyle{plain}
\newtheorem{theorem}{Theorem}[section]
\newtheorem{corollary}[theorem]{Corollary}
\newtheorem{lemma}[theorem]{Lemma}
\newtheorem{proposition}[theorem]{Proposition}
\theoremstyle{definition}
\newtheorem{definition}[theorem]{Definition}
\newtheorem{remark}[theorem]{Remark}
\usepackage{geometry}
\numberwithin{equation}{section} 
\usepackage{tikz}
\usepackage{caption}
\usepackage{graphicx}
\usepackage[titletoc,toc,title]{appendix}
\usepackage{lipsum}

\newcommand\blfootnote[1]{%
  \begingroup
  \renewcommand\thefootnote{}\footnote{#1}%
  \addtocounter{footnote}{-1}%
  \endgroup
}

\usepackage[colorlinks, citecolor=blue,pagebackref,hypertexnames=false]{hyperref}

\newcounter{comcount}

\begin{document}
\title[Isoperimetric Inequality on (QD) Manifolds]{Isoperimetric Inequality on Manifolds with Quadratically Decaying Curvature} 
\date{}
\author{Dangyang He}  %\sffamily
\address{Department of Mathematics, Sun Yat-sen University, Guangzhou 510275, Guangdong, P.R. China}
\email{hedy28@mail.sysu.edu.cn}

\begin{abstract}
In this paper, we investigate the reverse-improvement problem for Sobolev inequalities on manifolds with quadratically decaying Ricci curvature. Under volume comparison and relative connectedness of the ends, we show that an $L^2$ Sobolev inequality, together with an $L^p$ Hardy inequality, implies the corresponding $L^p$ Sobolev inequality, including the endpoint $p=1$. We also give a semigroup gluing proof of the two-scale isoperimetric estimate on connected sums with product ends, and an alternative proof of the classical Sobolev inequalities for generalized Grushin gradients. 
\end{abstract}

\maketitle

\tableofcontents

\blfootnote{$\textit{2020 Mathematics Subject Classification.}$ 53C20, 53C17}

\blfootnote{$\textit{Keywords and Phrases.}$ Isoperimetric inequality, Sobolev inequality, Hardy's inequality, Ricci curvature, Grushin spaces.}

\section{Introduction}

The isoperimetric inequality constitutes a fundamental result in geometry and analysis. In its classical planar formulation, it asserts that, among all simple closed curves of prescribed perimeter, the circle encloses the largest possible area. Equivalently, for any planar region of prescribed area, the circle achieves the least perimeter. This principle extends naturally to higher dimensions: among all bodies in $\mathbb{R}^n$ of fixed volume, the ball has the smallest surface area, and conversely, among all bodies with a given surface area, the ball encloses the greatest volume.

All manifolds in this paper are connected and without boundary. Let $M$ be a complete noncompact Riemannian manifold with measure $d\mu = dx = d\textrm{vol}$. Let $n>1$ be some parameter. We say the classical isoperimetric inequality holds on $M$ if
\begin{equation}\tag{$\textrm{ISO}_n$}\label{ISO}
    |\Omega|^{\frac{n-1}{n}} \le  C |\partial \Omega|,
\end{equation}
where $\Omega \subset M$ is any bounded domain with smooth boundary. In the above notation, $|\Omega|$ refers to the volume of the set $\Omega$ and $|\partial \Omega|$ denotes the standard Riemannian surface measure of its boundary. We call $n$ the optimal large-volume isoperimetric dimension when it is the largest exponent for which \eqref{ISO} holds for all sufficiently large sets, i.e., $|\Omega|\gtrsim1$. Note that on $\mathbb{R}^n$, the topological dimension coincides with its isoperimetric dimension. 

By the co-area formula, it is well known that \eqref{ISO} is equivalent to the Sobolev inequality (see \cite{Herbert}):
\begin{equation}\tag{$\textrm{S}_n^1$}\label{S_n^1}
    \left(\int_M |f|^{\frac{n}{n-1}} dx\right)^{\frac{n-1}{n}} \le C \int_M |\nabla f| dx
\end{equation}
for all $f\in C_c^\infty(M)$. More generally, one studies $L^p$ Sobolev inequalities
\begin{equation}\tag{$\textrm{S}_n^p$}\label{S_n^p}
    \left(\int_M |f|^{p^*} dx \right)^{1/p^*} \le C \left(\int_M |\nabla f|^p dx \right)^{1/p},
\end{equation}
where $p^*=\frac{np}{n-p}$ and $1<p<n$.

A simple but noteworthy feature of this family of inequalities is that they exhibit a self-improvement property. Specifically, if $\nabla$ is a local gradient, then from \eqref{S_n^p}, one can derive $(\textrm{S}_n^r)$ for any $1\le p<r<n$; see, for example, \cite{CSV}. Moreover, it is possible to extend this property to the cases $p=n$ and $p>n$ by considering Trudinger--Moser type and Gagliardo--Nirenberg type inequalities, respectively; see \cite{Coulhon2}.

However, in this article, our focus shifts to a reverse problem, which we refer to as the \textit{reverse-improvement} problem. Originally posed by Coulhon in \cite{Coulhon_Sobolev}, it asks under what conditions one can derive \eqref{ISO} from $(\textrm{S}_n^2)$. Throughout the paper, $\Delta=-\operatorname{div}\nabla$ denotes the nonnegative Laplace--Beltrami operator, and its heat semigroup is $e^{-t\Delta}$. In the diffusion calculations for the Grushin operator, we use the nonpositive generator $\mathcal{L}=-L$. The background of this problem can be traced to earlier results from \cite{Varo}, \cite{Nash}, and \cite{CKS}, which show that the uniform decay of the heat kernel is equivalent to the Sobolev inequality $(\textrm{S}_n^2)$ for $n>2$ or a Nash-type inequality for $n>1$. Concretely,
\begin{align*}
    \|e^{-t\Delta}\|_{1\to \infty} \le C t^{-n/2} \iff \begin{cases}
        (\textrm{S}_n^2), & n>2,\\
            \|f\|_2^{1+2/n} \le C\|\nabla f\|_2 \|f\|_1^{2/n}, & n>1.
    \end{cases}
\end{align*}
Hence, the \textit{reverse-improvement} problem can be restated as follows: under what conditions does uniform decay of the heat kernel imply the isoperimetric inequality?

From the perspective of curvature, the argument in \cite{Varopoulos_smalltime} shows that this \textit{reverse-improvement} property holds on manifolds with nonnegative Ricci curvature. Furthermore, if the Ricci curvature is bounded below by some constant, the property still holds but only in a local sense; see also \cite{Saloff_Sobolev}. Recall that we say $M$ satisfies the volume doubling condition if for any $x\in M$ and $R\ge r>0$, and for some $\mu>0$,
\begin{equation}\tag{$\textrm{D}$}\label{Doubling}
    \frac{V(x,R)}{V(x,r)} \le C \left(\frac{R}{r}\right)^\mu,
\end{equation}
where $B(x,r)$ denotes the geodesic ball centered at $x$ with radius $r$ and $V(x,r) = \mu(B(x,r))$ denotes the volume of the ball. It is also known that \eqref{Doubling} implies the so-called reverse doubling property (see \cite[Lemma~2.10]{GS_2005})
\begin{align}\tag{$\textrm{RD}_\nu$}\label{RD}
    \left(\frac{R}{r}\right)^\nu \le C \frac{V(x,R)}{V(x,r)},
\end{align}
for some $0<\nu\le \mu$. Denote by $\nabla$ the Riemannian gradient operator on $M$. We say $M$ satisfies the $L^p$ Poincaré inequality $(1\le p<\infty)$ if for every ball $B=B(x,r)$ and $f\in C^{\infty}(B)$,
\begin{align}\tag{$\textrm{P}_p$}\label{Pp}
    \int_{B}|f-f_B|^p dx \le C r^p \int_B |\nabla f|^p dx,\qquad f_B=\frac{1}{|B|}\int_B f\,dx.
\end{align}
Next, through a purely analytical approach, Coulhon and Saloff-Coste (see \cite{Coulhon-Saloff}) verify that this reverse implication holds when $M$ satisfies \eqref{Doubling} and a so-called pseudo-Poincaré inequality; see also \cite{Saloff_aspect}. In particular, their argument implies
\begin{align}\label{poincare_method}
    \eqref{Doubling} + \eqref{Pp} + (\textrm{S}_n^2) \implies \eqref{S_n^p}.
\end{align}
Our study begins by examining a specific Riemannian manifold: the connected sum of two copies of $\mathbb{R}^n$. Formally, we consider the \textit{reverse-improvement} problem on 
\begin{align*}
    M = \mathbb{R}^n \# \mathbb{R}^n.
\end{align*}
This manifold is obtained by gluing two copies of $\mathbb{R}^n \setminus B(0,1)$ along a smooth compact set; see \cite{GS} for a detailed description. It is known that $M$ satisfies \eqref{Doubling} and $(\textrm{S}_n^2)$ ($n\ge 3$), but \eqref{Pp} fails for every $1\le p\le n$. Moreover, $M$ does not need to have non-negative Ricci curvature everywhere. Consequently, neither the curvature method \cite{Varopoulos_smalltime} nor the Poincaré technique \cite{Saloff_aspect} can be applied directly in this case. 

To go further, let us first fix some notation. Denote by $e^{-t\Delta}(x,y)$ the Schwartz kernel of the heat semigroup. We use \eqref{DUE} to represent the following on-diagonal upper estimate for the heat kernel: 
\begin{equation}\tag{DUE}\label{DUE}
    e^{-t\Delta}(x,x) \le \frac{C}{V(x,\sqrt{t})}, \quad \forall x\in M,\quad t>0.
\end{equation}
By \cite{grigor_heatkernel_Gaussian}, under \eqref{Doubling}, the above on-diagonal bound self-improves to the off-diagonal upper bound
\begin{equation}\tag{UE}\label{UE}
    e^{-t\Delta}(x,y) \le \frac{C}{V(x,\sqrt{t})} e^{-\frac{d(x,y)^2}{ct}}. 
\end{equation}
We also recall the following heat kernel regularity condition:
\begin{align}\tag{$\textrm{G}_p$}\label{G_p}
    \|\sqrt{t} \nabla e^{-t\Delta} \|_{p\to p} \le C,\quad \forall t>0.
\end{align}
Theorem~9.1 of \cite{BCLS} and the comments that follow provide a general method for proving weak-type Sobolev inequalities. For completeness, our first proposition records this established criterion in the form used below.

\begin{proposition}\label{thm_anti_selfimprovement_general}
Let $M$ be a complete non-compact Riemannian manifold. Let $1\le p<2$ and $n>2$. Then, 
\begin{align*}
    (\textrm{S}_n^2) + (\textrm{G}_{p'}) \implies \eqref{S_n^p}.
\end{align*}
\end{proposition}

Note that this result generalizes the one from \cite{Varopoulos_smalltime}, because it is well known that $(\textrm{G}_{\infty})$ holds on manifolds with nonnegative Ricci curvature. Returning to our model case $M = \mathbb{R}^n \# \mathbb{R}^n$, it was shown in \cite[Theorem~1.6 and Example~4, Page 45]{CJKS} that \eqref{G_p} holds on $M$ if and only if $1<p<n$. Together with the standard increase of the Sobolev exponent (Lemma~\ref{lemma_self_improvement}), Proposition~\ref{thm_anti_selfimprovement_general} implies that \eqref{S_n^p} holds on $M$ for $p\in (n',n)$. This compact geometric perturbation can affect analytic properties. For instance, unlike the Euclidean setting, the Riesz transform on $M$ is $L^p$-bounded only on a finite interval; see \cite{CCH}. The classical Sobolev inequalities on this model are stable under the compact gluing; the limitation here is that the global gradient criterion does not recover the endpoint. This motivates a localized version of the semigroup argument.

We next recall another strategy for studying the \textit{reverse-improvement} problem. The first step is to recall a result from \cite{Varo}: for $n>2$,
\begin{align*}
    (\textrm{S}_n^2) \implies \|\Delta^{-1/2}\|_{p\to p^*} \lesssim 1,\quad \forall 1<p<n,\quad p^* = \frac{np}{n-p}.
\end{align*}
Combining this with the reverse Riesz inequality
\begin{align}\tag{$\textrm{RR}_p$}\label{RRp}
    \|\Delta^{1/2}f\|_p \le C \|\nabla f\|_p,\quad \forall f\in C_c^\infty(M).
\end{align}
gives another criterion for the \textit{reverse-improvement} problem:
\begin{align}\label{RR_method}
    (\textrm{S}_n^2) + \eqref{RRp} \implies \eqref{S_n^p}.
\end{align}
Now, it was proved in \cite{H2} that \eqref{RRp} holds on $M$ for all $1<p<\infty$. So, the above criterion indicates that \eqref{S_n^p} holds on $M$ for all $1<p<n$. %Note that this criterion improve the inequality \eqref{S_n^p} on $M$ to $1<p<n$. 
However, proving \eqref{RRp} may itself require a separate argument, and this criterion does not cover the endpoint $p=1$.

To study further in this direction, let us introduce the following notation. Let $o\in M$ be fixed. Denote by $r(x)$ the distance $d(x,o)$. In the rest of this note, we say $M$ has quadratically decaying Ricci curvature if for all $x\in M$, and some $\kappa \in \mathbb{R}$,
\begin{equation}\tag{QD}\label{QD}
    \textrm{Ric}_x \ge - \frac{\kappa^2}{\left(1+r(x) \right)^2} g_x,
\end{equation}
where $g$ is the Riemannian metric. Connected sums of finitely many Euclidean spaces, and manifolds with conical ends, are examples satisfying \eqref{QD}.

Suppose $M$ has at most finitely many ends. We say $M$ satisfies $(\textrm{RCE})$ if there exists a constant $\theta\in (0,1)$ such that for all $R\ge 1$ and all $x\in \partial B(o,R)$, there is a continuous path $\gamma:[0,1]\to B(o,\theta^{-1}R) \setminus B(o,\theta R)$ and a geodesic ray $\tau:[0,\infty)\to M\setminus B(o,R)$ such that
\begin{align*}
    \bullet \gamma(0)=x, \gamma(1)=\tau(0),\quad \bullet \textrm{length}(\gamma)\le R/\theta.
\end{align*}
This $(\textrm{RCE})$ condition is a generalization of the so-called $(\textrm{RCA})$ condition introduced in \cite{GS_2005} to manifolds with finitely many ends. We use the annular formulation in \cite[Definition~2.3]{Gilles}; see also \cite{GS_2005}.

For any $R\ge 1$ and any $x\in \partial B(o,R)$, we say $M$ satisfies volume comparison condition if 
\begin{equation}\tag{VC}\label{VC}
    V(o,R) \le C V(x,R/2).
\end{equation}
It was verified by Carron \cite[Theorem~2.4]{Gilles} and \cite{Gri_heatkernel} that 
\begin{align*}
    \eqref{QD} + \eqref{VC} + (\textrm{RCE}) \implies \eqref{Doubling} + \eqref{DUE}.
\end{align*}

Recall the following Hardy-type inequality
\begin{align}\tag{$\textrm{H}_p$}\label{Hp}
    \int_M \left|\frac{f(x)}{r(x)}\right|^p dx \le C \int_M |\nabla f(x)|^p dx,\quad \forall f\in C_c^\infty(M).
\end{align}

Our main contribution is the following localization criterion. Unlike Proposition~\ref{thm_anti_selfimprovement_general}, it does not require a global gradient estimate $(\textrm{G}_{p'})$. The new point in the argument is that the remote-ball scale is chosen in terms of the level parameter of the weak-type estimate. The Hardy inequality then controls the derivatives of the partition of unity and the remaining anchored piece. The passage from weak to strong type is the established truncation principle of \cite{BCLS}.

\begin{theorem}\label{ISO_QD}
Let $n>2$, and let $M$ be a complete non-compact Riemannian manifold satisfying \eqref{VC}, \eqref{QD}, and $(\textrm{RCE})$. For each $1\le p<n$,
\begin{align*}
    (\textrm{S}_n^2) + \eqref{Hp} \implies \eqref{S_n^p}.
\end{align*}

\end{theorem}

Corollary~\ref{Cor_ISO_Rn} below recovers the classical fact that \eqref{ISO} (or equivalently ($\textrm{S}_n^1$)) holds on $M = \mathbb{R}^n \# \mathbb{R}^n$. In the above result, the usual Poincaré inequality \eqref{Pp} (see \cite{Coulhon-Saloff, Saloff_aspect}) is replaced by a Hardy-type inequality \eqref{Hp}. The condition \eqref{Hp} is a gluing tool that combines the estimates on remote balls. The assumption $(\textrm S_n^2)$ with $n>2$ can be replaced by a uniform heat-kernel bound or a Nash-type inequality (see Remark~\ref{remark_n>2} below). The only reason for presenting the theorem in this way is to maintain uniformity in the notation.

\bigskip
Recall that the optimal large-volume isoperimetric dimension is the largest exponent $n$ for which \eqref{ISO} holds for all sufficiently large sets. Usually, this dimension $n$ is expected to coincide with the topological dimension, for instance $\mathbb{R}^n$. However, in the general setting, these two dimensions do not need to be the same, and the inequality \eqref{ISO} can be spoiled by very simple perturbations. For instance, one considers the cylinder $\mathbb{R}^n\times \mathbb{S}^{m-n}$ for some $m>n$, where $\mathbb{S}^{m-n}$ is the unit sphere of $\mathbb{R}^{m-n+1}$; see \cite{CGL}. Then \eqref{ISO} with exponent $m$ fails for large sets, whereas the large-volume isoperimetric dimension is $n$. The small-volume exponent is still determined by $m$. In other words, the topological dimension may not reflect the 'real' geometric properties on a general Riemannian manifold.

To study the isoperimetric problem further, instead of establishing \eqref{ISO}, one can also consider a generalized isoperimetric inequality associated with a function $\Phi$, i.e., the inequality: 
\begin{equation}\tag{$\textrm{ISO}_{\Phi}$}\label{GISO}
    |\Omega|\le C \Phi(|\Omega|) |\partial \Omega|,
\end{equation}
where $\Phi$, defined on $(0,\infty)$, is non-negative and non-decreasing. From \cite{Coulhon-Saloff}, it is well known that \eqref{GISO} holds on $M$ provided $M$ satisfies \eqref{Doubling} and Poincaré inequality $(\textrm{P}_1)$, where
\begin{equation*}
\Phi(t):= \inf \{r>0; V(r)\ge t\},
\end{equation*}
and $V(r) = \inf_{x\in M} V(x,r)$, is assumed to be strictly positive.

We next consider a class of connected sums studied in \cite{GS}, which generalizes the model $\mathbb{R}^n \# \mathbb{R}^n$. We use the notation of \cite{GS}. We consider connected sums:
\begin{align}\label{eq_manifold}
    M = (\mathbb{R}^{n_1}\times M_1)\# \dots \#(\mathbb{R}^{n_l}\times M_l) \quad n_*\ge 3 \quad l\ge 2,
\end{align}
where for all $1\le i\le l$, $M_i$ is a compact connected manifold without boundary with $n_i+\textrm{dim}(M_i)=N$ and $n_*=\min_i n_i$. We set $E_i:=\mathbb{R}^{n_i}\times M_i \setminus K_i$, where $K_i\subset \mathbb{R}^{n_i}\times M_i$ is a compact set, and denote the central connection part by $K$. That is
\begin{equation*}
    M = K \cup (\cup_i E_i).
\end{equation*}
When the dimensions $n_i$ are not all equal, $M$ is non-doubling. A bottleneck also prevents the global Poincar\'e inequality \eqref{Pp} for $1\le p\le n_*$. These product ends need not satisfy \eqref{QD}: a negative Ricci direction in a compact factor persists along the whole end. Thus this example is treated separately, using the product heat semigroups rather than a global curvature assumption.

The two-scale product estimates are part of the established theory of isoperimetric profiles; see \cite{CGL}. For a single compact--Euclidean product, the large-volume minimizing regions are known more precisely; see \cite{RV2017}. We give the following result as an application of our semigroup gluing method. Its conclusion concerns the profile up to multiplicative constants, rather than sharp constants or the existence of minimizing regions.

\begin{theorem}\label{thm_ISO_manifoldswithends}
Let $M$ be a manifold with ends defined by \eqref{eq_manifold}. Then \eqref{GISO} holds on $M$ with
\begin{equation*}
    \Phi(t) = \begin{cases}
        C_1 t^{1/N}, & t\le 1,\\
        C_2 t^{1/n_*}, & t\ge 1.
    \end{cases},
\end{equation*}
where $n_* = \min_i n_i$ and $N=n_i + \textrm{dim}(M_i)$.
\end{theorem}

It is clear that each end of $M$, defined in \eqref{eq_manifold}, has isoperimetric dimension $n_i$ (see Corollary~\ref{cor_ISO_S} below). However, Theorem~\ref{thm_ISO_manifoldswithends} shows that $M$ itself has isoperimetric dimension $n_*$. Unlike Theorem~\ref{ISO_QD}, where a Hardy-type inequality \eqref{Hp} is assumed, we develop an alternative 'gluing' tool; see Lemma~\ref{lemma_ISO_M} below.

\bigskip

Our final application concerns Sobolev and isoperimetric inequalities for generalized Grushin gradients. These inequalities must be distinguished from the existence of isoperimetric sets, which requires attainment of the perimeter infimum at prescribed volume. We prove a perimeter lower bound, not an attainment theorem.

Let $\beta\ge0$. For $n,m\ge1$, consider the quadratic form
\begin{align*}
 Q(f)=\int_{\mathbb{R}^{n+m}}|\nabla_Lf(\xi)|^2\,d\xi,
 \qquad f\in C_c^\infty(\mathbb{R}^{n+m}),
\end{align*}
where
\begin{align}\label{eq_grushin_operator_gradient}
 \nabla_L=(\nabla_x,|x|^\beta\nabla_y).
\end{align}
Its closure defines the nonnegative self-adjoint generalized Grushin operator
\begin{align}\label{eq_grushin_operator}
 L=-\sum_{i=1}^n\partial_{x_i}^2-|x|^{2\beta}\sum_{j=1}^m\partial_{y_j}^2
   =\Delta_x+|x|^{2\beta}\Delta_y;
\end{align}
see \cite{RS}. Its homogeneous dimension is $\mathcal{Q}=n+m(\beta+1)$.
For a bounded smooth domain $\Omega$, the associated perimeter is
\begin{align*}
 P_L(\Omega)=\int_{\partial\Omega}
 \bigl(|\nu_x|^2+|x|^{2\beta}|\nu_y|^2\bigr)^{1/2}\,d\mathcal{H}^{n+m-1},
\end{align*}
where $\nu=(\nu_x,\nu_y)$ is the Euclidean unit normal. The co-area formula identifies the $p=1$ Sobolev inequality with
$|\Omega|^{(\mathcal{Q}-1)/\mathcal{Q}}\lesssim P_L(\Omega)$.

The embedding theory for nonsmooth vector fields developed by Franchi and Lanconelli \cite{FranchiLanconelli1984}, and the later treatment of $\Delta_\lambda$-operators in \cite{KogojLanconelli2012}, already provides the background for the critical Grushin Sobolev exponent. Weighted-gradient Sobolev inequalities and extremals were studied by Monti \cite{Monti2006}. Endpoint geometric Sobolev and isoperimetric inequalities also appear in the work of Capogna, Danielli, and Garofalo \cite{CapognaDanielliGarofalo1994}, whose framework explicitly includes Baouendi--Grushin operators with a positive real degeneracy exponent. These works form the background for the inequalities used below, for both integer and real $\beta$.

More specialized results address minimizers: \cite{MM} treats the Grushin plane, while \cite{FM} treats $n=1$, $m\ge1$, and the class of $x$-spherically symmetric sets when $n\ge2$. These variational results answer a different question from the perimeter estimate proved here. For positive integer $\beta$, smooth sum-of-squares and Riesz-transform methods also yield familiar Sobolev consequences; see \cite{DS3}. For arbitrary real $\beta>0$, the nonsmooth-vector-field literature above is essential and cannot be replaced by that integer-parameter argument.

We therefore present the following classical inequalities as an alternative proof by remote-ball localization and Hardy gluing, without claiming a new range of parameters. The restriction $n\ge2$ comes from the endpoint Hardy inequality in the $x$-variables used in this proof. The analytic point is to justify a local heat-kernel gradient bound away from $\{x=0\}$ for the weighted diffusion, and then to localize in $x$ while retaining the full $y$-space.

\begin{theorem}\label{thm_ISO_Grushin}
Let $L$ be the operator defined in \eqref{eq_grushin_operator}, and let $\nabla_{L}$ be its associated gradient as in \eqref{eq_grushin_operator_gradient}. Suppose $n\ge 2$, $m\ge 1$ and $\beta > 0$. Then the Sobolev inequality $(\textrm{S}_{\mathcal{Q}}^p)$ holds on $\mathbb{R}^{n+m}$ for all $1\le p < \mathcal{Q}$, where $\mathcal{Q}:= n+m(\beta+1)$ is the homogeneous dimension of $L$. That is,
\begin{align*}
    \left(\int_{\mathbb{R}^{n+m}} |f(\xi)|^{\frac{\mathcal{Q}p}{\mathcal{Q}-p}} d\xi \right)^{\frac{\mathcal{Q}-p}{\mathcal{Q}}} \le C \int_{\mathbb{R}^{n+m}} |\nabla_{L}f(\xi)|^p d\xi,\quad \forall 1\le p < \mathcal{Q}
\end{align*}
for all $f\in C_c^\infty(\mathbb{R}^{n+m})$.
\end{theorem}

\section{Preliminaries: weak-type Sobolev inequality}
In this section, we introduce one of the key ingredients used in our argument. Let $(E,\mu)$ be a measure space with a non-negative $\sigma$-finite measure $\mu$. Let $\mathcal{F}$ be a class of functions on $E$ and $W$ be a norm or semi-norm on $\mathcal{F}$. We require that $\mathcal{F}$ is stable under truncations: if $f\in \mathcal{F}$, then $(f-t)^+ \wedge s \in \mathcal{F}$ for all $t,s\ge 0$. We assume that the test functions belong to every $L^p$ space used below; bounded functions with support of finite measure have this property. For instance, on a Riemannian manifold $M$ with Riemannian measure $\mu$, one can take $\mathcal{F}$ to be the space of Lipschitz functions with compact support. For $k\in \mathbb{Z}$ and $f\in \mathcal{F}$ with $f\ge 0$, we define the truncated function
\begin{align*}
    f_k := (f-2^k)^+ \wedge 2^k.
\end{align*}
Let $0<q\le \infty$. We say the functional $W$ satisfies condition $(\textrm{W}_q)$ if 
\begin{align}\tag{$\textrm{W}_q$}\label{eq_Wq}
    \left(\sum_{k\in \mathbb{Z}}W(f_k)^q\right)^{\frac{1}{q}} \le C_q W(f)
\end{align}
for all $f\in \mathcal{F}$. In particular, as also mentioned in \cite[Section~2]{BCLS}, the seminorm $W(f) = \left(\int_M |\nabla f|^p d\mu\right)^{\frac{1}{p}}$, where $\nabla$ denotes the Riemannian gradient, satisfies \eqref{eq_Wq} for all $q\in [p,\infty]$. We use the condition

\begin{align}\tag{$\textrm{S}_{r_0,s_0}^{*,\theta_0}$}\label{eq_weakSobolev}
    \sup_{\lambda>0} \left( \lambda \mu \left(\{|f|\ge \lambda\}\right)^{\frac{1}{r_0}} \right) \le C W(f)^{\theta_0} \left(\|f\|_\infty \left[\mu(\textrm{supp}(f))\right]^{\frac{1}{s_0}}\right)^{1-\theta_0}.
\end{align}
The following is one of the main results in \cite{BCLS}.
 
\begin{proposition}\cite[Theorem~3.1, Proposition~3.5]{BCLS}\label{le_BCLS}
Assume \eqref{eq_weakSobolev} holds for some $r_0,s_0\in (0,\infty]$ and $\theta_0 \in (0,1]$ and that the parameter $q = q(r_0,s_0,\theta_0)$ defined by
\begin{align*}
    \frac{1}{r_0}  = \frac{\theta_0}{q} + \frac{1-\theta_0}{s_0},
\end{align*}
satisfies $0< q<\infty$. Suppose $W$ satisfies \eqref{eq_Wq}. Then, for any $r,s\in (0,\infty]$, $\theta\in (0,1]$ such that 
\begin{align*}
\frac{1}{r} = \frac{\theta}{q} + \frac{1-\theta}{s},
\end{align*}
we have
\begin{align*}
    \|f\|_r \le C W(f)^\theta \|f\|_s^{1-\theta},
\end{align*}
for all $f\in \mathcal{F}$.
\end{proposition}

\begin{remark}
In particular, by letting $W(f) = \|\nabla f\|_p$ and taking $\mathcal{F}$ to be the space of Lipschitz functions with compact support, Proposition~\ref{le_BCLS} indicates that \eqref{eq_weakSobolev} implies the Sobolev inequality
\begin{align*}
    \|f\|_q \le C \|\nabla f\|_p,\quad q = \frac{r_0 s_0 \theta_0}{s_0 - r_0 (1-\theta_0)},
\end{align*}
as long as $0<p\le q <\infty$, with the displayed expression interpreted through its reciprocal identity if $r_0$ or $s_0$ is infinite.
\end{remark}

We record the following standard 'self-imrprovement' property of Sobolev inequalities for convenience; see \cite{CSV} for details.

\begin{lemma}\label{lemma_self_improvement}
Let $d>1$ and $1\le a<p<d$. Suppose that the gradient satisfies the chain rule and that $(\textrm{S}_d^a)$ holds for compactly supported test functions. Then $(\textrm{S}_d^p)$ holds as well. This applies both to $\nabla$ and to the Grushin gradient $\nabla_L$.
\end{lemma}

\begin{proof}
Put $a^*=da/(d-a)$, $p^*=dp/(d-p)$, and
$\gamma=p^*/a^*=p(d-a)/(a(d-p))>1$. Apply $(\textrm{S}_d^a)$ to $|f|^\gamma$, using smooth approximation at its zeros if necessary. The chain rule and H\"older's inequality give
\begin{align*}
 \|f\|_{p^*}^{\gamma}
 &\le C\gamma\bigl\||f|^{\gamma-1}\nabla f\bigr\|_a\\
 &\le C\gamma\|\nabla f\|_p\,
       \|f\|_{(\gamma-1)ap/(p-a)}^{\gamma-1}
  = C\gamma\|\nabla f\|_p\|f\|_{p^*}^{\gamma-1}.
\end{align*}
Here $(\gamma-1)ap/(p-a)=p^*$. Cancelling the last factor proves the result. The proof uses only the chain rule, so it also holds with $\nabla_L$.
\end{proof}

\section{Via heat kernel regularity}\label{section_2}

In this section, we give a proof of Proposition~\ref{thm_anti_selfimprovement_general}, which is a criterion for \textit{reverse improvement} of Sobolev inequalities. This result is implicit in \cite[Theorem~9.1]{BCLS}. In \cite{BCLS}, the authors proved that a family of weak-type Sobolev inequalities implies the corresponding strong-type estimate. In particular, in \cite[Theorem~9.1]{BCLS} (also see \cite{Coulhon-Saloff}), they gave a general method to prove a weak-type Sobolev inequality. We include the heat-semigroup argument to make the subsequent localization transparent.

\begin{remark}\label{remark_n>2}
We mention that the assumption $n>2$ in the statement is not essential. As can be seen from the proof below, $(\textrm{S}_n^2)$ is used only to obtain heat-kernel decay. In the case where $n>1$, one may replace the assumption $(\textrm{S}_n^2)$ by a Nash-type inequality:
\begin{align*}
    \|f\|_2^{1+2/n} \le C \|\nabla f\|_2 \|f\|_1^{2/n},
\end{align*}
which is well known to be equivalent to $\|e^{-t\Delta}\|_{1\to \infty}\le Ct^{-n/2}$, see \cite[Section~2]{CKS} and \cite{Nash}. Thus, for $1\le p<\min\{2,n\}$, the implication $(\textrm{G}_{p'})\implies \eqref{S_n^p}$ holds as long as $\|e^{-t\Delta}\|_{1\to \infty}\lesssim t^{-n/2}$ for all $t>0$.
\end{remark}

\begin{proof}[Proof of Proposition~\ref{thm_anti_selfimprovement_general}]
Let $f\in C_c^\infty(M)$ and $\lambda>0$. Set $q=p^*=\frac{np}{n-p}$. By Proposition~\ref{le_BCLS}, one only needs to establish \eqref{eq_weakSobolev}:
\begin{align}\label{SS}
    \sup_{\lambda>0} \lambda \textrm{vol}\left(\{|f|\ge \lambda\}\right)^{\frac{1}{r_0}} \le C \|\nabla f\|_p^{\theta_0} \left(\|f\|_\infty \left[\textrm{vol}(\textrm{supp}(f))\right]^{\frac{1}{s_0}}\right)^{1-\theta_0},
\end{align}
for some $r_0,s_0\in (0,\infty]$, $\theta_0\in (0,1]$ such that
\begin{align*}
    \frac{1}{r_0} = \frac{\theta_0}{q} + \frac{1-\theta_0}{s_0},
\end{align*}
and the strong type estimate \eqref{S_n^p} follows immediately. 

We follow the ideas from \cite[Theorem~9.1]{BCLS} closely. One writes
\begin{equation*}
    \textrm{vol} \left( \{ |f|\ge \lambda \} \right) \le \textrm{vol} \left( \{ |f-e^{-t\Delta}f|\ge \lambda/2 \} \right) + \textrm{vol} \left( \{ |e^{-t\Delta}f|\ge \lambda/2 \} \right).
\end{equation*}
By $(\textrm{S}_n^2)$ and \cite{Varopoulos_smalltime}, $\|e^{-t\Delta}f\|_{\infty}\le C t^{-n/2}\|f\|_1$. It then follows by interpolation that $\|e^{-t\Delta}\|_{q\to \infty}\le C t^{-\frac{n}{2q}}$. Therefore,
\begin{align*}
    \|e^{-t\Delta}f\|_\infty \le C t^{-\frac{n}{2q}} \|f\|_q \le C t^{-\frac{n}{2q}} \textrm{vol}\left(\textrm{supp}(f)\right)^{\frac{1}{q}}\|f\|_\infty.
\end{align*}
Choose
\begin{equation*}
    t = \left(\frac{2C\,\textrm{vol}\left(\textrm{supp}(f)\right)^{1/q}\|f\|_\infty}{\lambda}\right)^{2q/n},
\end{equation*}
We obtain
\begin{equation*}
    \textrm{vol} \left( \{ |f|\ge \lambda \} \right) \le \textrm{vol} \left( \{ |f-e^{-t\Delta}f|\ge \lambda/2 \} \right).
\end{equation*}
Note that $f-e^{-t\Delta}f = \int_0^t \Delta e^{-s\Delta}f ds$. One continues with
\begin{align*}
    \textrm{vol} \left( \{ |f|\ge \lambda \} \right) \le C \lambda^{-p} \left[ \int_0^t \|\Delta e^{-s\Delta}f\|_{p} ds \right]^p.
\end{align*}
Observe that by duality, we have for $f,g\in C_c^\infty(M)$,
\begin{align*}
    \|\Delta e^{-s\Delta}f\|_{p} = \sup_{\|g\|_{p'}=1} |\langle \Delta e^{-s\Delta}f, g \rangle| = \sup_{\|g\|_{p'}=1} |\langle \nabla f, \nabla e^{-s\Delta}g \rangle| \le \|\nabla f\|_p \|\nabla e^{-s\Delta}\|_{p' \to p'}.
\end{align*}
By $(\textrm{G}_{p'})$,
\begin{align*}
    \textrm{vol} \left( \{ |f|\ge \lambda \} \right) &\le C \lambda^{-p} \left[ \int_0^t \|\nabla f\|_p s^{-1/2}  ds \right]^p
\end{align*}
or equivalently
\begin{align*}
    \lambda \textrm{vol} \left( \{ |f|\ge \lambda \} \right)^{\frac{1}{p}} &\le C\|\nabla f\|_p \int_0^t  s^{-1/2}  ds \\
    &\sim \|\nabla f\|_p t^{1/2}\sim  \|\nabla f\|_p \left(\frac{\textrm{vol}\left(\textrm{supp}(f)\right)^{\frac{1}{q}}\|f\|_\infty}{\lambda}\right)^{q/n}.
\end{align*}
Denoting $\theta = \frac{q}{n}$, one gets
\begin{align*}
    \lambda \textrm{vol} \left( \{ |f|\ge \lambda \} \right)^{\frac{1}{p(1+\theta)}} \lesssim \|\nabla f\|_p^{\frac{1}{1+\theta}} \left[\textrm{vol}\left(\textrm{supp}(f)\right)^{\frac{1}{q}}\|f\|_\infty\right]^{\frac{\theta}{1+\theta}}.
\end{align*}
Finally, by setting $s_0 = q$, $r_0 = p(1+\theta)$ and $\theta_0 = \frac{1}{1+\theta}$, one concludes
\begin{align*}
    \sup_{\lambda>0} \lambda \textrm{vol} \left( \{ |f|\ge \lambda \} \right)^{\frac{1}{r_0}} \lesssim \|\nabla f\|_p^{\theta_0} \left[\textrm{vol}\left(\textrm{supp}(f)\right)^{\frac{1}{s_0}}\|f\|_\infty\right]^{1-\theta_0}
\end{align*}
which is nothing but \eqref{SS} with $r_0,s_0 \in (0,\infty]$ and $\theta_0\in (0,1]$. Now, by Proposition~\ref{le_BCLS}, \eqref{S_n^p} holds on $M$ provided
\begin{equation*}
    \frac{1}{r_0} = \frac{\theta_0}{q} + \frac{1-\theta_0}{s_0} = \frac{1}{q}.
\end{equation*}
Indeed, recall that $\frac{1}{q} = \frac{1}{p} - \frac{1}{n}$. Hence, $1+\theta = \frac{q}{p}$ and
\begin{align*}
    r_0 = p(1+\theta) = q
\end{align*}
as desired.

\end{proof}

\begin{remark}
It is worth mentioning that this \textit{reverse-improvement} property also holds if one assumes the following domination property:
\begin{align}\label{eq_domination}
    |\nabla e^{-t\Delta}f(x)|^2 \le C e^{-t\Delta}\left(|\nabla f|^2\right)(x),\quad \forall f\in C_c^\infty(M),\quad \forall t>0.
\end{align}
Indeed, it has been shown in \cite[Theorem~4.1]{CD} that under the assumption of domination property, 
\begin{align*}
    \|\nabla f\|_p^2 \le C_p \|f\|_p \|\Delta f\|_p,\quad \forall f\in C_c^\infty(M),\quad \forall 1<p<\infty.
\end{align*}
Then it follows by \cite[Proposition~3.6]{Coulhon-Sikora} that this multiplicative inequality is equivalent to \eqref{G_p} for all $p\in (1,\infty)$. As a consequence of Proposition~\ref{thm_anti_selfimprovement_general}, the conjunction of $(\textrm{S}_n^2)$ and \eqref{eq_domination} implies \eqref{S_n^p} for $1<p<2$; the exponents $2\le p<n$ then follow from Lemma~\ref{lemma_self_improvement}.
\end{remark}

The following corollary is immediate by splitting the proof into cases $t\le 1$ and $t\ge 1$.

\begin{corollary}\label{cor_ISO_S}
Let $M$ be a complete non-compact Riemannian manifold such that $(\textrm{G}_\infty)$ holds and
\begin{align*}
    \|e^{-t\Delta}\|_{1\to \infty}\lesssim \begin{cases}
        t^{-m/2}, & 0<t\le 1,\\
        t^{-n/2}, & t\ge 1,
    \end{cases}
\end{align*}
for some $m\ge n > 1$. Then \eqref{GISO} holds on $M$ with
\begin{align*}
    \Phi(t) = \begin{cases}
        c_1 t^{1/m}, & t\le 1,\\
        c_2 t^{1/n}, & t\ge 1.
    \end{cases}
\end{align*}
\end{corollary}

\begin{proof}
By duality and $(\textrm{G}_\infty)$,
\begin{align*}
 \|f-e^{-t\Delta}f\|_1
 \le\int_0^t\|\Delta e^{-s\Delta}f\|_1\,ds
 \lesssim\sqrt t\,\|\nabla f\|_1.
\end{align*}
Let $v=|\Omega|$ and approximate $\chi_\Omega$ by smooth functions. Choose $t$ so that
$\|e^{-t\Delta}\|_{1\to\infty}v\le1/2$. The assumed bounds allow
$\sqrt t\lesssim v^{1/m}$ for $v\le1$ and
$\sqrt t\lesssim v^{1/n}$ for $v\ge1$, after enlarging the constants on a fixed intermediate interval. On $\Omega$, $\chi_\Omega-e^{-t\Delta}\chi_\Omega\ge1/2$, so the preceding inequality yields
$v\lesssim\sqrt t\,|\partial\Omega|$. This is the claimed estimate.
\end{proof}

Next, by a simple argument, Proposition~\ref{thm_anti_selfimprovement_general} guarantees the following result. 

\begin{corollary}\label{cor_QD_general}
Let $M$ be a complete non-compact Riemannian manifold satisfying \eqref{Doubling}, \eqref{DUE}, \eqref{QD} and \eqref{RD} for some $\nu>2$, and let $n>2$. Then,
\begin{align*}
    (\textrm{S}_n^2) \implies \eqref{S_n^p}
\end{align*}
for all $\nu'<p<2$.
\end{corollary}

\begin{proof}[Proof of Corollary~\ref{cor_QD_general}]
For a target exponent $p\in(\nu',2)$, set $a=p'\in(2,\nu)$ and choose $b$ with $a<b<\nu$. Proposition~1.10 of \cite{ACDH} applied at the kernel exponent $b$ gives $(\textrm G_a)$, once the estimate below is known with $b$ in place of its dummy exponent $p$. Proposition~\ref{thm_anti_selfimprovement_general} then gives the desired Sobolev inequality. Thus it suffices to verify
\begin{align*}
    \left\| \nabla_x e^{-t\Delta}(\cdot,y) \right\|_p^p \lesssim \frac{1}{t^{p/2} V(y,\sqrt{t})^{p-1}},\quad 2<p<\nu
\end{align*}
for all $t>0$ and $y\in M$. In fact, this property has been mentioned in other articles before. Here, we give a direct proof for the sake of completeness. 

By \cite[Section~3.2, 3.3]{Gilles}, we have
\begin{align*}
    \int_M |\nabla e^{-t\Delta}(x,y)|^p dx \lesssim &\int_{r(x)\le \sqrt{t}} r(x)^{-p} V(x,\sqrt{t})^{-p} e^{-\frac{d(x,y)^2}{ct}} dx \\
    &+ \int_{r(x)\ge \sqrt{t}} t^{-p/2} V(x,\sqrt{t})^{-p} e^{-\frac{d(x,y)^2}{ct}} dx:= I + II.
\end{align*}
For $II$, by \eqref{Doubling} and \cite[Lemma~2.1]{CD2}, we have
\begin{align*}
    II\lesssim t^{-p/2} V(y,\sqrt{t})^{-p} \int_M \left(1+\frac{d(x,y)}{\sqrt{t}}\right)^{p\mu} e^{-\frac{d(x,y)^2}{ct}} dx \lesssim \frac{1}{t^{p/2}V(y,\sqrt{t})^{p-1}}.
\end{align*}
For $I$, we first consider the case where $r(y) \le 2\sqrt{t}$. Note that by again \eqref{Doubling},
\begin{align*}
    I \lesssim V(y,\sqrt{t})^{-p} \int_{r(x)\le \sqrt{t}} r(x)^{-p} e^{-\frac{d(x,y)^2}{c' t}} dx \le V(y,\sqrt{t})^{-p} \int_{r(x)\le \sqrt{t}} r(x)^{-p} dx.
\end{align*}
Let $dV$ be the Lebesgue--Stieltjes measure associated with $V(r)=V(o,r)$. The above integral equals
\begin{align*}
    \int_0^{\sqrt{t}} r^{-p} dV(r) \lesssim \int_0^{\sqrt{t}} t^{-p/2} \left(\frac{V(\sqrt{t})}{V(r)}\right)^{p/\nu} dV(r) \sim t^{-p/2} V(o,\sqrt{t}),
\end{align*}
where the inequality follows by \eqref{RD}, i.e. $\frac{V(\sqrt{t})}{V(r)}\gtrsim \left(\frac{\sqrt{t}}{r}\right)^\nu$.

Therefore, for $r(y) \le 2\sqrt{t}$,
\begin{align*}
    I\lesssim \frac{V(o,\sqrt{t})}{t^{p/2}V(y,\sqrt{t})^{p}}\lesssim \frac{V(y, r(y)+\sqrt{t})}{t^{p/2}V(y,\sqrt{t})^{p}}\lesssim \frac{1}{t^{p/2}V(y,\sqrt{t})^{p-1}}.
\end{align*}
Next, if $\frac{r(y)}{2}\ge \sqrt{t}\ge r(x)$, then $d(x,y)\ge \frac{r(y)}{2}$. We obtain by \eqref{Doubling}
\begin{align*}
    I&\lesssim V(y,\sqrt{t})^{-p} e^{-\frac{r(y)^2}{ct}}\int_{r(x)\le \sqrt{t}} r(x)^{-p} dx \\
    &\lesssim \frac{1}{t^{p/2}V(y,\sqrt{t})^{p-1}} \frac{V(y, r(y)+\sqrt{t})}{V(y,\sqrt{t})} e^{-\frac{r(y)^2}{ct}}\\
    &\lesssim \frac{1}{t^{p/2}V(y,\sqrt{t})^{p-1}} \left(1+\frac{r(y)}{\sqrt{t}}\right)^\mu e^{-\frac{r(y)^2}{ct}}\\
    &\lesssim \frac{1}{t^{p/2}V(y,\sqrt{t})^{p-1}} 
\end{align*}
as desired.

\end{proof}

\begin{remark}
It is worth comparing the above corollary with Theorem~\ref{ISO_QD}. Indeed, as mentioned in the introduction, the assumptions \eqref{VC}, \eqref{QD} and $(\textrm{RCE})$ imply \eqref{Doubling} and \eqref{DUE}. Under \eqref{Doubling}, \eqref{DUE}, and \eqref{QD}, the following criteria are therefore available:
\begin{enumerate}
    \item The proof of Theorem~\ref{ISO_QD} gives that $(\textrm{S}_n^2)$ together with \eqref{Hp} implies \eqref{S_n^p} for all $1 \le p < 2$.
    \item Corollary~\ref{cor_QD_general} guarantees that $(\textrm{S}_n^2)$ together with \eqref{RD} (with $\nu>2$) implies \eqref{S_n^p} for all $\nu' < p < 2$.
\end{enumerate}
\end{remark}

\section{Manifolds with (QD) curvature}\label{sec4}

In this section, we consider manifolds with quadratically decaying Ricci curvature and give a Hardy-type criterion for the implication $(\textrm{S}_n^2)\implies\eqref{ISO}$. The argument uses gradient estimates only on remote balls and only up to the time scale determined by the level set under consideration. In particular, it does not require a global $(\textrm{G}_\infty)$ estimate.

\subsection{Doubling case: proof of Theorem~\ref{ISO_QD}}

\begin{definition}
Let $o\in M$. We say that a ball $B=B(x,r(B))$ is remote if
\[
    r(B)\le \frac{r(x)}{2}=\frac{d(x,o)}{2}.
\]
\end{definition}

We start with a covering argument from \cite[Section~4.3]{Gilles}; see also \cite[Section~2.1]{DR_hardy}. We record the dependence on the central radius, since this radius will vary in the proof.

\begin{lemma}\label{lemma_remoteball}
Suppose that $M$ satisfies \eqref{Doubling}. For every $R>0$, there is a locally finite covering $\{B_\alpha=B(x_\alpha,r_\alpha)\}_{\alpha\ge0}$ and a subordinate Lipschitz partition of unity $\{\mathcal X_\alpha\}_{\alpha\ge0}$ with the following properties:
\begin{enumerate}
    \item $B_0=B(o,R)$, and $B_\alpha$ is remote for $\alpha\ge1$;
    \item $\sum_{\alpha\ge0}\chi_{B_\alpha}\le c_0$, where $c_0$ is independent of $R$;
    \item for $\alpha\ge1$, $2^{-10}r(x_\alpha)\le r_\alpha\le2^{-9}r(x_\alpha)$ and $r_\alpha\ge2^{-10}R$;
    \item $0\le\mathcal X_\alpha\le1$, $\operatorname{supp}\mathcal X_\alpha\subset B_\alpha$, $\sum_{\alpha\ge0}\mathcal X_\alpha=1$, and $|\nabla\mathcal X_\alpha|\le C/r_\alpha$ almost everywhere for $\alpha\ge1$.
\end{enumerate}
The constants depend only on the doubling constant.
\end{lemma}

\begin{proof}
Set $A_N=\{x:R2^{N-1}\le r(x)\le R2^N\}$ for $N\ge0$. In each $A_N$, choose a maximal disjoint family of balls $B(x_{N,j},R2^{N-13})$ with centers in $A_N$, and let
\[
    B_{N,j}=B(x_{N,j},R2^{N-10}).
\]
The balls $\tfrac12 B_{N,j}$ cover $A_N$. Doubling and the disjointness of the smaller balls give a uniform bound on the number of centers in each annulus. Indeed, every smaller ball lies in $B(o,2R2^N)$, and its volume is bounded below by a fixed multiple of $V(o,2R2^N)$, by doubling and comparison of the centers. Thus each annulus contains at most a fixed number of centers. A point can belong only to balls from at most three consecutive annuli, which proves the uniform overlap bound after $B_0$ is added. The radius estimates in (3) follow directly from $R2^{N-1}\le r(x_{N,j})\le R2^N$.

Choose Lipschitz cutoffs $\psi_{N,j}$ supported in $B_{N,j}$, equal to one on $\tfrac12 B_{N,j}$, with $|\nabla\psi_{N,j}|\le C/(R2^{N-10})$. Choose $\psi_0$ supported in $B(o,R)$, equal to one on $B(o,3R/4)$, with $|\nabla\psi_0|\le C/R$. The sum $\Psi=\psi_0+\sum_{N,j}\psi_{N,j}$ satisfies $1\le\Psi\le C$. Normalize the cutoffs by $\mathcal X_\alpha=\psi_\alpha/\Psi$. Intersecting remote balls have comparable radii; if a remote ball meets the support of $\nabla\psi_0$, its radius is comparable to $R$ with fixed constants. The quotient rule therefore gives (4). Lipschitz cutoffs suffice below, since the energy identity and the Sobolev estimates extend to compactly supported Lipschitz functions by approximation.
\end{proof}

The next observation is the local heat-kernel estimate needed in the proof.

\begin{lemma}\label{ISO_QD_Lemma1}
Let $M$ satisfy \eqref{Doubling}, \eqref{DUE}, and \eqref{QD}, and let $B_\alpha$, $\alpha\ge1$, be the remote balls in Lemma~\ref{lemma_remoteball}. For every fixed $A>0$,
\begin{equation}\label{improve}
    \|\sqrt{s}\,\nabla e^{-s\Delta}\|_{L^\infty(B_\alpha)\to L^\infty(B_\alpha)}\le C_A,
    \qquad 0<s\le A r_\alpha^2.
\end{equation}
Here and below, an input on $B_\alpha$ is extended by zero to $M$, and the output is restricted to $B_\alpha$. The constants are independent of $\alpha$ and $R$.
\end{lemma}

\begin{proof}
Under \eqref{Doubling}, \eqref{DUE}, and \eqref{QD}, the local Li--Yau argument gives
\begin{equation}\label{eq_QD_gradient_kernel}
    |\nabla_x e^{-s\Delta}(x,y)|
    \le C\left(s^{-1/2}+r(x)^{-1}\right)
       \frac{1}{V(x,\sqrt{s})}
       \exp\left(-\frac{d(x,y)^2}{cs}\right),\qquad x\ne o.
\end{equation}
See \cite[Sections~3.2--3.3]{Gilles} and \cite[Lemma~2.4]{DR_hardy}. In this application, the heat kernel solves the heat equation on every remote ball, and \eqref{QD} gives a Ricci lower bound of order $-r(x)^{-2}$ on a fixed enlargement of that ball. The Gaussian upper estimate supplied by \eqref{Doubling} and \eqref{DUE} then gives \eqref{eq_QD_gradient_kernel} by localization.

For $x\in B_\alpha$, we have $r(x)\simeq r_\alpha$. Hence, for $g$ supported in $B_\alpha$,
\begin{align}\label{eq_s-1/2,1}
    |\nabla e^{-s\Delta}g(x)|
    &\le C\|g\|_\infty\left(s^{-1/2}+r_\alpha^{-1}\right)
       \frac{1}{V(x,\sqrt{s})}
       \int_M\exp\left(-\frac{d(x,y)^2}{cs}\right)dy\nonumber\\
    &\le C\left(s^{-1/2}+r_\alpha^{-1}\right)\|g\|_\infty.
\end{align}
The last step is the usual annular summation under doubling. If $s\le A r_\alpha^2$, then $r_\alpha^{-1}\le\sqrt A\,s^{-1/2}$, which proves \eqref{improve}.
\end{proof}

\begin{remark}
The time restriction in \eqref{improve} is essential to the present argument. An estimate for all $s>0$ also follows if \eqref{RD} holds with $\nu\ge1$. Indeed, for $\sqrt{s}>r_\alpha$, retaining the support of $g$ in the integral gives
\begin{equation}\label{eq_s-1/2,2}
    |\nabla e^{-s\Delta}g(x)|
    \lesssim r_\alpha^{-1}\frac{V(x_\alpha,r_\alpha)}{V(x,\sqrt{s})}\|g\|_\infty
    \lesssim r_\alpha^{-1}\left(\frac{r_\alpha}{\sqrt{s}}\right)^\nu\|g\|_\infty
    \le s^{-1/2}\|g\|_\infty.
\end{equation}
This additional reverse-doubling hypothesis is not needed in Theorem~\ref{ISO_QD}.
\end{remark}

We now combine the local gradient estimate with Hardy's inequality. Quadratic curvature decay does not imply that $M$ is asymptotically Euclidean; what is used here is the uniform Ricci lower bound after rescaling a remote ball.

\begin{proof}[Proof of Theorem~\ref{ISO_QD}]
We first treat $1\le p<2$. Let $\lambda>0$, $q=np/(n-p)$, and $f\in C_c^\infty(M)$ with $f\ne0$. Set
\[
    R=\left(\frac{\|f\|_q}{\lambda}\right)^{q/n},
\]
and use the covering and partition of unity from Lemma~\ref{lemma_remoteball}. Write $f_\alpha=f\mathcal X_\alpha$. The overlap bound implies
\[
    \operatorname{vol}\{|f|\ge\lambda\}
    \le\sum_{\alpha\ge1}\operatorname{vol}\{x\in B_\alpha:|f_\alpha(x)|\ge\lambda/c_0\}
       +\operatorname{vol}\{x\in B_0:|f_0(x)|\ge\lambda/c_0\}
    =:\sum_{\alpha\ge1}\mathcal J_\alpha+\mathcal J_0.
\]
The assumption $(\textrm S_n^2)$ gives $\|e^{-t\Delta}\|_{1\to\infty}\lesssim t^{-n/2}$. Interpolation with the $L^\infty$ contraction property gives
\[
    \|e^{-t\Delta}f_\alpha\|_\infty\le C t^{-n/(2q)}\|f\|_q.
\]
Choose $t=(4c_0C\|f\|_q/\lambda)^{2q/n}=cR^2$. Then the preceding norm is at most $\lambda/(4c_0)$, so that
\begin{equation}\label{eq_I_alpha}
    \mathcal J_\alpha
    \lesssim\lambda^{-p}\left(\int_0^t
       \|\Delta e^{-s\Delta}f_\alpha\|_{L^p(B_\alpha)}ds\right)^p.
\end{equation}
Here we used $f_\alpha-e^{-t\Delta}f_\alpha=\int_0^t\Delta e^{-s\Delta}f_\alpha\,ds$.

For a test function $g$ supported in $B_\alpha$, self-adjointness and the energy identity give
\begin{align}\label{eq_dual}
    |\langle\Delta e^{-s\Delta}f_\alpha,g\rangle|
    &=|\langle\nabla f_\alpha,\nabla e^{-s\Delta}g\rangle|
    \le\|\nabla f_\alpha\|_p
       \|\nabla e^{-s\Delta}g\|_{L^{p'}(B_\alpha)}.
\end{align}
By construction, $r_\alpha\ge2^{-10}R$, so that $t\le c2^{20}r_\alpha^2$. Lemma~\ref{ISO_QD_Lemma1} therefore applies for every $0<s\le t$. Interpolating \eqref{improve} with
\[
    \|\sqrt{s}\,\nabla e^{-s\Delta}\|_{L^2(B_\alpha)\to L^2(B_\alpha)}\le C,
\]
which follows from the spectral theorem on $M$, gives
\[
    \|\nabla e^{-s\Delta}\|_{L^{p'}(B_\alpha)\to L^{p'}(B_\alpha)}\lesssim s^{-1/2},
    \qquad 0<s\le t.
\]
This includes $p=1$, where $p'=\infty$. Combining \eqref{eq_I_alpha} and \eqref{eq_dual}, we obtain
\[
    \mathcal J_\alpha\lesssim\lambda^{-p}t^{p/2}\|\nabla f_\alpha\|_p^p
    \lesssim\lambda^{-p-pq/n}\|f\|_q^{pq/n}\|\nabla f_\alpha\|_p^p.
\]
Since $r(x)\simeq r_\alpha$ on $B_\alpha$, finite overlap and \eqref{Hp} give
\begin{align}\label{eq_ISO_QD_3}
    \sum_{\alpha\ge1}\mathcal J_\alpha
    &\lesssim\lambda^{-p-pq/n}\|f\|_q^{pq/n}
       \left(\|\nabla f\|_p^p+\|f/r\|_p^p\right)\nonumber\\
    &\lesssim\lambda^{-p-pq/n}\|f\|_q^{pq/n}\|\nabla f\|_p^p.
\end{align}
For the central term, Chebyshev's inequality and \eqref{Hp} yield
\[
    \mathcal J_0\lesssim\lambda^{-p}\|f\|_{L^p(B(o,R))}^p
    \le\lambda^{-p}R^p\|f/r\|_p^p
    \lesssim\lambda^{-p-pq/n}\|f\|_q^{pq/n}\|\nabla f\|_p^p.
\]
Consequently,
\[
    \lambda^{1+q/n}\operatorname{vol}\{|f|\ge\lambda\}^{1/p}
    \lesssim\|\nabla f\|_p\|f\|_q^{q/n}.
\]
Since $1+q/n=q/p$ and $\|f\|_q\le\|f\|_\infty\operatorname{vol}(\operatorname{supp}f)^{1/q}$, this is \eqref{eq_weakSobolev} with
\[
    r_0=s_0=q,\qquad\theta_0=p/q,\qquad W(f)=\|\nabla f\|_p.
\]
The truncation condition \eqref{eq_Wq} holds because $q>p$. Proposition~\ref{le_BCLS} now gives $(\textrm S_n^p)$.

For $p=2$, the conclusion is the assumed inequality $(\textrm S_n^2)$. For $2<p<n$, it follows directly from $(\textrm S_n^2)$ by the standard power argument, without any Hardy hypothesis. More explicitly, put $a=p(n-2)/(2(n-p))$ and $q=np/(n-p)$. Applying $(\textrm S_n^2)$ to $|f|^a$ and using H\"older's inequality gives
\[
    \|f\|_q^a\lesssim a\big\||f|^{a-1}\nabla f\big\|_2
    \le a\|f\|_q^{a-1}\|\nabla f\|_p,
\]
because $2an/(n-2)=q$ and $2(a-1)p/(p-2)=q$. Approximation justifies the power at zeros of $f$, and cancellation proves the desired inequality. Thus all exponents in the statement have been covered.
\end{proof}

Recall that in the following corollary, bounded geometry means a uniform lower bound for the Ricci curvature and a positive lower bound for the injectivity radius.

\begin{corollary}\label{cor_ISO_QD_1}
Let $M$ be a complete $n$-dimensional Riemannian manifold of bounded geometry, where $n>2$. Suppose that $M$ satisfies \eqref{VC}, \eqref{QD}, $(\textrm{RCE})$, and \eqref{RD} with some $\nu>1$. Then
\[
    (\textrm S_n^2)\iff\eqref{ISO}.
\]
\end{corollary}

\begin{proof}
The implication from \eqref{ISO} to $(\textrm S_n^2)$ follows from the co-area formula and the power argument in Lemma~\ref{lemma_self_improvement}. We prove the converse.

The assumptions \eqref{VC}, \eqref{QD}, and $(\textrm{RCE})$ imply \eqref{Doubling} and \eqref{DUE}. Moreover, \eqref{QD} gives uniform $L^1$ Poincar\'e inequalities on remote balls, and ordinary local Poincar\'e inequalities on a fixed central region. Thus \cite[Theorem~2.3]{DR}, applied with $p=1<\nu$, gives the inhomogeneous Hardy inequality
\begin{equation}\label{eq_inhomogeneous_Hardy_one}
    \int_M\frac{|f(x)|}{1+r(x)}dx\lesssim\int_M|\nabla f(x)|dx,
    \qquad f\in C_c^\infty(M).
\end{equation}
We use this inequality only at scales at least one.

Let $f\in C_c^\infty(M)$ and set $R=\max\{1,\|f\|_1\}^{1/n}$. Use Lemma~\ref{lemma_remoteball} at this radius and put $\lambda=p=q=1$ in the level-set argument above. Choose $t=cR^2$ so that $\|e^{-t\Delta}f_\alpha\|_\infty\le1/(4c_0)$. Since $r_\alpha\ge2^{-10}R$, the same restricted-time gradient estimate gives
\[
    \sum_{\alpha\ge1}\mathcal J_\alpha
    \lesssim R\sum_{\alpha\ge1}\|\nabla f_\alpha\|_1
    \lesssim R\left(\|\nabla f\|_1+\|f/(1+r)\|_1\right)
    \lesssim R\|\nabla f\|_1.
\]
In the second inequality we used $R\ge1$ and $r(x)\simeq r_\alpha\gtrsim R$ on a remote ball; the comparison constants are fixed. Similarly,
\[
    \mathcal J_0\lesssim\|f\|_{L^1(B(o,R))}
    \le(1+R)\|f/(1+r)\|_1
    \lesssim R\|\nabla f\|_1.
\]
It follows that
\[
    \operatorname{vol}\{|f|\ge1\}
    \lesssim\|\nabla f\|_1\max\{1,\|f\|_1\}^{1/n}.
\]
Apply this inequality to smooth approximations of $2\chi_\Omega$ for a bounded smooth domain with $|\Omega|\ge1$. The standard approximation of a smooth domain by tubular cutoffs gives convergence in $L^1$ and convergence of the gradient integrals to $2|\partial\Omega|$. We obtain $|\Omega|^{(n-1)/n}\lesssim|\partial\Omega|$ for such domains. For small volumes, bounded geometry supplies the local inequality
\[
 |\Omega|^{(n-1)/n}\le C\bigl(|\partial\Omega|+|\Omega|\bigr),
\]
so the volume term can be absorbed for $|\Omega|\le v_0$, with $v_0>0$ fixed. On the remaining interval $[v_0,1]$, the same level-set estimate applied to approximations of $2\chi_\Omega$ gives $|\Omega|\lesssim|\partial\Omega|$, and hence the required inequality. Together these prove \eqref{ISO}.
\end{proof}

\begin{remark}
The reverse-doubling assumption in Corollary~\ref{cor_ISO_QD_1} is used to obtain \eqref{eq_inhomogeneous_Hardy_one} from \cite[Theorem~2.3]{DR}. It can be replaced by \eqref{eq_inhomogeneous_Hardy_one} itself. Doubling alone gives a reverse-doubling exponent $\nu>0$, which need not be greater than one; that observation does not justify this use of the Hardy theorem.
\end{remark}

\begin{corollary}\label{Cor_ISO_Rn}
Let $M=\mathbb R^n\#\mathbb R^n$, where $n\ge2$. Then $(\textrm S_n^p)$ holds on $M$ for every $1\le p<n$.
\end{corollary}

\begin{proof}
This is the compact-gluing argument for Euclidean ends given after Lemma~\ref{lemma_ISO_M}; that argument also applies in dimension two. It gives $(\textrm S_n^1)$, and Lemma~\ref{lemma_self_improvement} yields the remaining exponents. For $n>2$, one can alternatively apply Corollary~\ref{cor_ISO_QD_1}: this connected sum has bounded geometry, satisfies \eqref{VC}, \eqref{QD}, $(\textrm{RCE})$ and $(\textrm{RD}_n)$, and satisfies $(\textrm S_n^2)$.
\end{proof}

\subsection{A non-doubling example: proof of Theorem~\ref{thm_ISO_manifoldswithends}}

The main purpose of this section is to illustrate the weak-type method on the manifolds with ends defined by \eqref{eq_manifold}. The two-scale isoperimetric profile of a product $\mathbb{R}^{n_i}\times M_i$ is classical; see \cite{CGL}. For the unperturbed products, substantially stronger results identifying the large isoperimetric regions are available in \cite{RV2017}. Here we give a semigroup proof of the corresponding inequality after gluing finitely many such ends across a compact set. The contribution of this example is the compact-core gluing argument, rather than a new product embedding or a classification of minimizing sets.

We first recall the following estimate on the complete product spaces. Only the case $n_i\ge 3$ assumed in \eqref{eq_manifold} will be used below.

\begin{lemma}[{\cite[Lemma~2.2]{SSi}}]\label{lemma_SSi}
For each $1\le i\le l$, let $\Delta_i$ be the nonnegative Laplace--Beltrami operator on $\mathbb{R}^{n_i}\times M_i$. Then
\begin{align*}
    \|\sqrt{t}\nabla e^{-t\Delta_i}\|_{\infty\to\infty}\lesssim 1,
    \qquad t>0.
\end{align*}
\end{lemma}

Fix a relatively compact neighbourhood $\widetilde K$ of the central set $K$, large enough to contain all the cutoff derivatives used below. The following estimate controls the errors introduced by these cutoffs.

\begin{lemma}\label{lemma_ISO_M}
Under the assumptions of Theorem~\ref{thm_ISO_manifoldswithends},
\begin{align*}
    \|u\|_{L^1(\widetilde K)}\lesssim\|\nabla u\|_1,
    \qquad u\in C_c^\infty(M).
\end{align*}
\end{lemma}

\begin{proof}
Choose a fixed nonnegative function $v\in C_c^\infty(M)$ such that $v=1$ on $\widetilde K$. By \cite[Lemma~2.7]{HS}, since $n_*\ge3$, there exists a smooth function $U$ satisfying
\begin{align*}
    \Delta U=v,\qquad \|\nabla U\|_\infty<\infty.
\end{align*}
Indeed, the zero-energy solution in that lemma has bounded gradient on compact subsets and gradient of order $O(r^{1-n_i})$ on the $i$th end. We apply the lemma only to this one fixed smooth function $v$; no uniform estimate over arbitrary bounded right-hand sides is needed.

Integration by parts gives
\begin{align*}
    \int_{\widetilde K}|u|\,d\mu
    &\le\int_M v|u|\,d\mu
      =\int_M\langle\nabla U,\nabla|u|\rangle\,d\mu\\
    &\le\|\nabla U\|_\infty\|\nabla u\|_1.
\end{align*}
The identity for $|u|$ follows by first using the smooth compactly supported functions $(|u|^2+\epsilon^2)^{1/2}-\epsilon$ and then letting $\epsilon\downarrow0$.
\end{proof}

\begin{proposition}\label{lemma_ISO_large}
Let $M$ be a manifold with ends defined by \eqref{eq_manifold}. For every bounded smooth domain $\Omega\subset M$ with $|\Omega|\ge1$,
\begin{align*}
    |\Omega|^{\frac{n_*-1}{n_*}}\lesssim|\partial\Omega|.
\end{align*}
\end{proposition}

\begin{proof}
Choose smooth functions $0\le\phi_i\le1$ such that $\operatorname{supp}\phi_i\subset E_i$ and $\phi_i=1$ outside a compact subset of the $i$th end. Since the ends are disjoint, $\sum_i\phi_i\le1$. For $f\in C_c^\infty(M)$ write
\begin{align*}
    f=\sum_{i=1}^l f\phi_i+f\Bigl(1-\sum_{i=1}^l\phi_i\Bigr)
      =\sum_{i=1}^l f_i+f_0.
\end{align*}
By zero extension, $f_i\in C_c^\infty(\mathbb{R}^{n_i}\times M_i)$, while $f_0\in C_c^\infty(\widetilde K)$. Put $A=\|f\|_1$. We shall prove
\begin{equation}\label{ISO_weakestimate}
    \mu\bigl(\{|f|\ge1\}\bigr)
    \lesssim\|\nabla f\|_1\max\{1,A\}^{1/n_*}.
\end{equation}
The decomposition gives
\begin{align*}
    \mu\bigl(\{|f|\ge1\}\bigr)
    &\le\sum_{i=1}^l\mu\bigl(\{x\in E_i:|f_i(x)|\ge1/(l+1)\}\bigr)\\
    &\quad+\mu\bigl(\{x\in\widetilde K:|f_0(x)|\ge1/(l+1)\}\bigr)
      =\sum_{i=1}^l\mathcal L_i+\mathcal L_0.
\end{align*}

The product heat kernel estimates in \cite[Section~2.2]{HS} imply
\begin{equation}\label{eq_1toinfty}
    \|e^{-t\Delta_i}\|_{1\to\infty}\le Ct^{-n_i/2},\qquad t\ge1.
\end{equation}
Choose a fixed sufficiently large $a\ge1$ and set
\begin{align*}
    t_i=\bigl(a\max\{1,A\}\bigr)^{2/n_i}.
\end{align*}
Then \eqref{eq_1toinfty} and $\|f_i\|_1\le A$ give
\begin{align*}
    \|e^{-t_i\Delta_i}f_i\|_\infty\le\frac1{2(l+1)}.
\end{align*}
Consequently, Chebyshev's inequality yields
\begin{align*}
    \mathcal L_i
    &\lesssim\|f_i-e^{-t_i\Delta_i}f_i\|_1
     \le\int_0^{t_i}\|\Delta_i e^{-s\Delta_i}f_i\|_1\,ds.
\end{align*}
By self-adjointness and integration by parts, for bounded test functions $g$,
\begin{align*}
    |\langle\Delta_i e^{-s\Delta_i}f_i,g\rangle|
    &=|\langle\nabla f_i,\nabla e^{-s\Delta_i}g\rangle|\\
    &\le\|\nabla f_i\|_1\|\nabla e^{-s\Delta_i}g\|_\infty.
\end{align*}
Duality and Lemma~\ref{lemma_SSi} therefore give
\begin{align}\label{eq_L_i}
    \mathcal L_i
    &\lesssim\|\nabla f_i\|_1\int_0^{t_i}s^{-1/2}\,ds\nonumber\\
    &\lesssim\bigl(\|\nabla f\|_1+\|f\|_{L^1(\widetilde K)}\bigr)
             \max\{1,A\}^{1/n_i}\nonumber\\
    &\lesssim\|\nabla f\|_1\max\{1,A\}^{1/n_*}.
\end{align}
In the last line we used Lemma~\ref{lemma_ISO_M}, $n_i\ge n_*$ and $\max\{1,A\}\ge1$. The semigroup integral can first be taken over $[\epsilon,t_i]$ and then extended to $[0,t_i]$ by the integrable bound $s^{-1/2}$ and strong continuity on $L^1$.

For the compact term, Chebyshev's inequality and Lemma~\ref{lemma_ISO_M} give
\begin{equation}\label{eq_L0}
    \mathcal L_0\lesssim\|f_0\|_1
      \lesssim\|f\|_{L^1(\widetilde K)}
      \lesssim\|\nabla f\|_1.
\end{equation}
This proves \eqref{ISO_weakestimate}.

Let $\Omega$ be a bounded smooth domain. Choose $0\le h_j\le1$ with $h_j\in C_c^\infty(M)$ such that $h_j\to\chi_\Omega$ in $L^1$ and almost everywhere, and
\begin{align*}
    \int_M|\nabla h_j|\,d\mu\longrightarrow|\partial\Omega|;
\end{align*}
see \cite[p.~34]{Hebey_Sobolev_Spaces_on_Riemannian_manifolds}. Apply \eqref{ISO_weakestimate} to $f=2h_j$ and use Fatou's lemma on $\{2h_j\ge1\}$. We obtain
\begin{equation}\label{ends_all_volume_auxiliary}
    |\Omega|\lesssim|\partial\Omega|\max\{1,|\Omega|\}^{1/n_*}.
\end{equation}
For $|\Omega|\ge1$ this is the stated inequality.
\end{proof}

To complete the proof of Theorem~\ref{thm_ISO_manifoldswithends}, it remains to treat small volumes. The manifold has bounded geometry, so its local Sobolev inequality, with uniform constants, gives
\begin{align*}
    |\Omega|^{(N-1)/N}\le C\bigl(|\partial\Omega|+|\Omega|\bigr).
\end{align*}
For $|\Omega|$ sufficiently small the volume term is absorbed into the left-hand side. For volumes in a fixed interval $[v_0,1]$, \eqref{ends_all_volume_auxiliary} gives $|\Omega|\lesssim|\partial\Omega|$, hence also
\begin{align*}
    |\Omega|^{(N-1)/N}\lesssim|\partial\Omega|.
\end{align*}
Together with Proposition~\ref{lemma_ISO_large}, this proves the theorem.

\begin{remark}\label{ends_profile_sharpness}
The exponents in Theorem~\ref{thm_ISO_manifoldswithends} are optimal in the sense of comparison up to multiplicative constants. More precisely, if
\begin{align*}
    I_M(v)=\inf\{\,|\partial\Omega|: \Omega\subset M
          \text{ is a bounded smooth domain},\ |\Omega|=v\,\},
\end{align*}
then
\begin{align*}
    I_M(v)\asymp
    \begin{cases}
        v^{(N-1)/N},&0<v\le1,\\
        v^{(n_*-1)/n_*},&v\ge1.
    \end{cases}
\end{align*}
The lower bounds are the theorem. Small coordinate balls give the upper bound for sufficiently small $v$. For large $v$, choose an end with $n_i=n_*$ and translate a product set $B_{\mathbb R^{n_*}}(z,R)\times M_i$ sufficiently far out that it lies entirely in $E_i$. Its volume and perimeter are
\begin{align*}
    \omega_{n_*}|M_i|R^{n_*},\qquad
    n_*\omega_{n_*}|M_i|R^{n_*-1},
\end{align*}
respectively. Choosing $R$ to prescribe the volume proves the upper bound; the remaining compact range of volumes is absorbed into the comparison constants. This is precisely the sharpness supplied by the smallest-dimensional end. It does not determine a sharp constant, a finer asymptotic expansion, or attainment of the profile.
\end{remark}

\begin{remark}
For $M=\mathbb R^n\#\mathbb R^n$, $n\ge2$, there is also a direct proof, including the two-dimensional case. Choose a compactly supported cutoff $\eta$ equal to one near $K$. The two restrictions of $(1-\eta)u$ extend by zero to compactly supported functions on $\mathbb R^n$. The Euclidean Sobolev inequality, together with the Sobolev inequality on a fixed relatively compact neighbourhood containing $\operatorname{supp}\eta$, gives
\begin{align*}
    \|u\|_{n/(n-1)}
    \lesssim\|\nabla u\|_1+\|u\|_{L^1(\widetilde K)}.
\end{align*}
For completeness, the compact term can be controlled without the nonparabolic Green function used in Lemma~\ref{lemma_ISO_M}. Fix an annulus $A=\{R_0<r<R_1\}$ contained in one Euclidean end. Since $u$ has compact support,
\begin{align*}
    |u(r\theta)|\le\int_r^\infty|\partial_su(s\theta)|\,ds.
\end{align*}
Integrating first over $R_0<r<R_1$ and then over $\theta\in\mathbb S^{n-1}$ gives
\begin{align*}
    \|u\|_{L^1(A)}\lesssim\|\nabla u\|_{L^1(E_i)}.
\end{align*}
Indeed, $s\ge R_0>0$ allows the unweighted radial integral to be bounded by a constant times the integral with measure $s^{n-1}\,ds$. Choose a connected relatively compact smooth domain $D$ containing both $\widetilde K$ and $A$. The $L^1$ Poincar\'e inequality on $D$, together with
\begin{align*}
    |A||u_D|\le\|u\|_{L^1(A)}+\|u-u_D\|_{L^1(D)},
\end{align*}
implies
\begin{align*}
    \|u\|_{L^1(\widetilde K)}
    \lesssim\|u\|_{L^1(A)}+\|\nabla u\|_{L^1(D)}
    \lesssim\|\nabla u\|_1.
\end{align*}
This absorbs the compact term and proves the claimed Sobolev inequality for all $n\ge2$.
\end{remark}

\section{Isoperimetric inequality on Grushin spaces}

In this section, we give a heat-semigroup proof of the Grushin Sobolev inequality stated in Theorem~\ref{thm_ISO_Grushin}. As explained in the introduction, the embedding itself is known; the purpose here is to illustrate the remote-ball and Hardy-gluing argument in a degenerate setting.
Recall that for $(x,y)\in\mathbb{R}^{n+m}$, the nonnegative Grushin operator is
\begin{align}\label{Grushin}
 L=-\sum_{i=1}^n\partial_{x_i}^2-|x|^{2\beta}\sum_{j=1}^m\partial_{y_j}^2
   =\Delta_x+|x|^{2\beta}\Delta_y,\qquad \beta\ge0.
\end{align}
Its associated gradient is
\begin{align}\label{Grushin_gradient}
 \nabla_L=(\nabla_x,|x|^\beta\nabla_y),
\end{align}
and $|\nabla_Lf|^2=|\nabla_xf|^2+|x|^{2\beta}|\nabla_yf|^2$.
The functional $W(f)=\|\nabla_Lf\|_p$ satisfies \eqref{eq_Wq} for all $q\ge p$, by the chain rule and the disjointness of the truncation layers.
For real-valued $f,g\in C_c^\infty(\mathbb{R}^{n+m})$, integration by parts gives
\[
 \int Lf\,g\,d\xi=\int\nabla_Lf\cdot\nabla_Lg\,d\xi.
\]
We use the intrinsic distance
\[
 d(\xi,\eta)=\sup\{|\phi(\xi)-\phi(\eta)|:
   \phi\in C(\mathbb{R}^{n+m})\cap D_{\mathrm{loc}}(Q),\ |\nabla_L\phi|\le1\},
\]
where $D_{\mathrm{loc}}(Q)$ is the local domain of the closed quadratic form associated with $L$. We write $B(\xi,r)$ for the corresponding ball and $V(\xi,r)$ for its Lebesgue measure.

\subsection{Curvature-dimension inequality and gradient estimates}

Throughout this section, $L$ is nonnegative and
\[
 \mathcal{L}=-L=\sum_i\partial_{x_i}^2+|x|^{2\beta}\sum_j\partial_{y_j}^2
\]
is the nonpositive diffusion generator. Thus $\partial_tu=\mathcal{L}u=-Lu$, and the heat semigroup is $e^{-tL}$. Following Bakry and \'Emery~\cite{BE}, we set
\begin{align}\label{carre1}
 \Gamma(f,g)&=\tfrac12\bigl(\mathcal{L}(fg)-f\mathcal{L}g-g\mathcal{L}f\bigr),\\
 \Gamma_2(f,g)&=\tfrac12\bigl(\mathcal{L}\Gamma(f,g)-\Gamma(f,\mathcal{L}g)-\Gamma(g,\mathcal{L}f)\bigr).
 \label{carre2}
\end{align}
We abbreviate $\Gamma(f,f)$ and $\Gamma_2(f,f)$ to $\Gamma(f)$ and $\Gamma_2(f)$. In particular, $\Gamma(f)=|\nabla_Lf|^2$. The curvature-dimension condition is
\begin{align}\tag{$\mathrm{CD}(K,N)$}\label{CD(K,N)}
 \Gamma_2(f)\ge K\Gamma(f)+\frac1N(\mathcal{L}f)^2.
\end{align}
Only a local version of this condition, away from $\{x=0\}$, will be used.

\begin{lemma}\label{le_Ric_Grushin}
Let $n,m\ge1$ and $\beta\ge0$. There exists $K_0\ge0$, depending only on $n,m,\beta$, such that
\[
 \Gamma_2(f)\ge-\frac{K_0}{|x|^2}\Gamma(f)
                  +\frac{1}{2(n+m)}(Lf)^2
\]
on $(\mathbb{R}^n\setminus\{0\})\times\mathbb{R}^m$ for every smooth $f$.
\end{lemma}

\begin{proof}
The calculation in Appendix~\ref{appendix} gives
\begin{align}\label{CD1}
 \Gamma_2(f)={}&S_{xx}+2|x|^{2\beta}S_{xy}+|x|^{4\beta}S_{yy}
 +\beta(n+2\beta-2)|x|^{2\beta-2}|\nabla_yf|^2\nonumber\\
 &+4\beta|x|^{2\beta-2}\sum_{i,j}x_if_{y_j}f_{x_i y_j}
 +2\beta|x|^{2\beta-2}\Delta_yf\sum_i x_if_{x_i},
\end{align}
where $S_{xx}=\sum_{i,k}f_{x_i x_k}^2$, $S_{xy}=\sum_{i,j}f_{x_i y_j}^2$ and $S_{yy}=\sum_{j,l}f_{y_j y_l}^2$. Recall that $\Delta_y=-\sum_j\partial_{y_j}^2$.
By Young's inequality,
\begin{align}\label{CD2}
 4\beta|x|^{2\beta-2}\sum_{i,j}x_if_{y_j}f_{x_i y_j}
 \ge -|x|^{2\beta}S_{xy}-4\beta^2|x|^{2\beta-2}|\nabla_yf|^2.
\end{align}
Since $|\Delta_yf|^2\le mS_{yy}$, the other mixed term satisfies
\begin{align}\label{CD3}
 2\beta|x|^{2\beta-2}\Delta_yf\sum_i x_if_{x_i}
 \ge-\tfrac12|x|^{4\beta}S_{yy}-2m\beta^2|x|^{-2}|\nabla_xf|^2.
\end{align}
Consequently,
\begin{align}\label{CD5}
 \Gamma_2(f)\ge \tfrac12\bigl(S_{xx}+2|x|^{2\beta}S_{xy}
                      +|x|^{4\beta}S_{yy}\bigr)-K_0|x|^{-2}\Gamma(f),
\end{align}
where, for example, $K_0=2m\beta^2+|\beta(n-2\beta-2)|$ is sufficient.
The symmetric matrix
\[
 \mathcal H=\begin{bmatrix}
 (f_{x_i x_k})&(|x|^\beta f_{x_i y_j})\\
 (|x|^\beta f_{y_j x_i})&(|x|^{2\beta}f_{y_j y_l})
 \end{bmatrix}
\]
has trace $-Lf$. The trace inequality therefore gives
\begin{align}\label{CD4}
 (Lf)^2\le(n+m)\bigl(S_{xx}+2|x|^{2\beta}S_{xy}+|x|^{4\beta}S_{yy}\bigr).
\end{align}
Combining \eqref{CD5} and \eqref{CD4} proves the assertion.
\end{proof}

We recall two results from Robinson and Sikora~\cite{RS}. We write $\xi=(x,y)$ and $\eta=(x',y')$.

\begin{lemma}[\cite{RS}, Proposition~5.1]\label{lemma_DS1}
Under the assumptions of Theorem~\ref{thm_ISO_Grushin},
\[
 d(\xi,\eta)\sim |x-x'|+
 \frac{|y-y'|}{(|x|+|x'|)^\beta+|y-y'|^{\beta/(\beta+1)}},
\]
and
\[
 V(\xi,r)\sim\begin{cases}
 r^{\mathcal Q},&r\ge |x|,\\
 r^{n+m}|x|^{m\beta},&r\le |x|.
 \end{cases}
\]
In particular, $V(\xi,sr)\le Cs^{\mathcal Q}V(\xi,r)$ for $s\ge1$ and $r>0$.
\end{lemma}

\begin{lemma}[\cite{RS}, Theorem~6.4 and Corollary~6.6]\label{le_RS}
The heat kernel of $L$ satisfies
\[
 0\le e^{-tL}(\xi,\eta)\le\frac{C}{V(\xi,\sqrt t)}
               \exp\!\left(-\frac{d(\xi,\eta)^2}{ct}\right)
\]
for $t>0$ and almost every $\xi,\eta\in\mathbb{R}^{n+m}$.
\end{lemma}

We shall also use the following elementary comparison. Here $B_n$ and $B_m$ denote Euclidean balls.

\begin{lemma}\label{lemma_G_volume}
There exist $c_1,c_2>0$, depending only on $n,m,\beta$, such that, for $x\ne0$,
\begin{align}\label{eq_lemma_volume_2}
 B_n(x,c_1|x|)\times B_m(y,c_1|x|^{\beta+1})
 \subset B((x,y),|x|)
 \subset B_n(x,c_2|x|)\times B_m(y,c_2|x|^{\beta+1}).
\end{align}
\end{lemma}

\begin{proof}
Let $r=|x|$. If $|x-x'|<c_1r$ and $|y-y'|<c_1r^{\beta+1}$, the distance estimate in Lemma~\ref{lemma_DS1} gives $d(\xi,\eta)\le C(2c_1r)<r$ after choosing $c_1$ sufficiently small. This proves the first inclusion.
Conversely, $d(\xi,\eta)<r$ implies $|x-x'|\le Cr$, so $|x|+|x'|\le Cr$. The second term in the distance estimate is comparable to
\[
 \min\!\left(\frac{|y-y'|}{(|x|+|x'|)^\beta},
                  |y-y'|^{1/(\beta+1)}\right).
\]
Its being at most $Cr$ implies $|y-y'|\le Cr^{\beta+1}$ in either case. This proves the second inclusion.
\end{proof}

\begin{definition}\label{def_x_remote}
A ball $B((x,y),r)$ is called \emph{$x$-remote} if $r\le c|x|$, where $c>0$ is sufficiently small and fixed. In particular, its fixed dilates, when used below, remain in $\{(x',y'):|x'|\sim |x|\}$.
\end{definition}

We next justify the local heat-kernel gradient estimate. Although the coefficients are degenerate at $\{x=0\}$, the following argument takes place entirely in a smooth weighted patch disjoint from that set.

\begin{lemma}\label{le_Grushin_gradient}
Under the assumptions of Theorem~\ref{thm_ISO_Grushin},
\[
 |\nabla_Le^{-tL}(\xi,\eta)|
 \le\left(\frac1{\sqrt t}+\frac1{|x|}\right)
       \frac{C}{V(\xi,\sqrt t)}\exp\!\left(-\frac{d(\xi,\eta)^2}{ct}\right)
\]
for $t>0$, $x\ne0$ and almost every $\eta$.
\end{lemma}

\begin{proof}
Fix $\xi_0=(x_0,y_0)$ with $r_0=|x_0|>0$. On every relatively compact patch away from $\{x=0\}$, the coefficients are smooth and uniformly elliptic. Interior parabolic regularity therefore makes the restriction of the global heat kernel smooth in $(t,\xi)$ for $t>0$ and almost every $\eta$. On $\{x\ne0\}$, set
\[
 g=dx^2+|x|^{-2\beta}dy^2,\qquad
 \varphi=-m\beta\log|x|.
\]
Then $d\xi=e^{-\varphi}d\mathrm{vol}_g$ and
\[
 \mathcal L=\operatorname{div}_g\nabla_g-
       \langle\nabla_g\varphi,\nabla_g\,\cdot\,\rangle,
 \qquad \Gamma(f)=|\nabla_gf|_g^2.
\]
Thus this weighted generator, is the one appearing in the heat equation.

Choose $\rho=c r_0$ with $c>0$ small enough. A smooth cutoff $0\le\chi\le1$, equal to one near $\xi_0$, can be chosen with compact support in
\[
 |x-x_0|<2\rho,\qquad |y-y_0|<2\rho r_0^\beta
\]
and with
\begin{align}\label{grushin_cutoff}
 \frac{\Gamma(\chi)}{\chi}\le C\rho^{-2},\qquad
 \mathcal L\chi\ge-C\rho^{-2}.
\end{align}
Here the quotient is understood to be zero where $\chi=0$. For example, take the square of a product of fixed smooth Euclidean cutoffs in $(x-x_0)/\rho$ and $(y-y_0)/(\rho r_0^\beta)$. On their support, $|x|\sim r_0$, which verifies both bounds directly from the expression for $\mathcal L$. Lemma~\ref{le_Ric_Grushin} gives $\mathrm{CD}(-K,N)$ there, where $N=2(n+m)$ and $K\le Cr_0^{-2}$.

For completeness, the local Li--Yau argument can be carried out using this cutoff alone. Let $u>0$ solve $\partial_tu=\mathcal Lu$ on this patch, put $h=\log u$, and fix $\alpha>1$. Set
\[
 F=t\bigl(\Gamma(h)-\alpha\partial_th\bigr),\qquad z=t\Gamma(h).
\]
The diffusion chain rule and \eqref{CD(K,N)} imply
\begin{align}\label{grushin_LY_evolution}
 (\mathcal L-\partial_t)F
 \ge\frac{2}{N\alpha^2t}\bigl(F+(\alpha-1)z\bigr)^2
       -2Kz-2\Gamma(h,F)-\frac Ft.
\end{align}
Indeed, $\mathcal Lh=\partial_th-\Gamma(h)$ and
$(\mathcal L-\partial_t)(\Gamma(h)-\alpha\partial_th)
=2\Gamma_2(h)-2\Gamma(h,\Gamma(h)-\alpha\partial_th)$.
Fix a time $T>0$. If $\chi F$ has a positive maximum on the time interval $[0,T]$, let $\tau\in(0,T]$ be its time coordinate and write $G=\chi F$ and $Z=\chi z$ there. Since $\nabla F=-F\nabla\chi/\chi$ there, multiplying the maximum-principle inequality by $\tau\chi$ and using \eqref{grushin_cutoff} gives
\[
 0\ge\frac{2}{N\alpha^2}\bigl(G+(\alpha-1)Z\bigr)^2
       -2K\tau Z-\left(1+\frac{C\tau}{\rho^2}\right)G
       -2G\sqrt{\frac{C\tau Z}{\rho^2}}.
\]
For sufficiently small $\varepsilon>0$, Young's inequality gives
\[
 2G\sqrt{C\tau Z/\rho^2}\le\varepsilon GZ+C_\varepsilon\tau\rho^{-2}G,
 \qquad
 2K\tau Z\le\varepsilon Z^2+C_\varepsilon K^2\tau^2.
\]
Absorbing these terms into the square yields
\[
 c_\alpha G^2\le C_\alpha\left(1+\frac{\tau}{\rho^2}\right)G
                         +C_\alpha K^2\tau^2.
\]
Hence the maximum is at most $C_\alpha(1+T/\rho^2+KT)$. At $(T,\xi_0)$, $\chi=1$, so division by $T$ proves, on renaming $T$ as $t$,
\begin{align}\label{le_grushin_grad1}
 \frac{|\nabla_Lu|^2}{u^2}(t,\xi_0)
       -\alpha\frac{\partial_tu}{u}(t,\xi_0)
 \le C_{\alpha,n,m,\beta}\left(\frac1t+\frac1{r_0^2}\right).
\end{align}
One may first apply this argument to $u(t+\varepsilon)$ and then let $\varepsilon\downarrow0$. The cutoff has compact support in the smooth patch, so the argument needs neither completeness of that patch nor boundary conditions for $u$ on its boundary. The global heat kernel is a local weak solution of the heat equation, and interior regularity on this smooth elliptic patch makes it smooth there for positive times. Thus the argument applies to its restriction.

The corresponding time-derivative estimate is
\begin{align}\label{le_grushin_grad2}
 |\partial_te^{-tL}(\xi,\eta)|
 \le\frac{C}{tV(\xi,\sqrt t)}
                  \exp\!\left(-\frac{d(\xi,\eta)^2}{ct}\right).
\end{align}
This is the standard time-derivative consequence of the Gaussian estimate for a symmetric strongly local form; see \cite[Theorem~2.6, Corollary~2.7 and the subsequent remark]{Sturm3}. In the present setting, the intrinsic metric is complete and induces the Euclidean topology by Lemma~\ref{lemma_DS1}; the closed form is regular and strongly local. Volume doubling follows from that lemma, and the scale-invariant local Sobolev inequality used in the cited result follows from doubling and the $L^2$-Poincar\'e inequality of \cite{DS2}. Thus its hypotheses hold uniformly.

Write $p_t=e^{-tL}(\xi_0,\eta)$ and
$A_t=CV(\xi_0,\sqrt t)^{-1}\exp(-d(\xi_0,\eta)^2/(ct))$,
enlarging the constants in Lemma~\ref{le_RS} and \eqref{le_grushin_grad2} if necessary. Multiply \eqref{le_grushin_grad1} by $p_t^2$ to obtain
\[
 |\nabla_Lp_t|^2
 \le\alpha p_t|\partial_tp_t|
           +C(t^{-1}+r_0^{-2})p_t^2
 \le C(t^{-1}+r_0^{-2})A_t^2.
\]
Taking square roots proves the estimate. If needed, apply the logarithmic argument to $p_t+\varepsilon$ and let $\varepsilon\downarrow0$.
\end{proof}

\subsection{Hardy's inequality}

The next ingredient is a Hardy inequality in the $x$-variables.

\begin{lemma}\label{le_grushin_hardy}
Let $n\ge2$, $m\ge1$ and $\beta\ge0$. For $1\le p<n$,
\[
 \left\|\frac{u}{|x|}\right\|_p
       \le\frac{p}{n-p}\|\nabla_xu\|_p
       \le\frac{p}{n-p}\|\nabla_Lu\|_p,
 \qquad u\in C_c^\infty(\mathbb{R}^{n+m}).
\]
\end{lemma}

\begin{proof}
For each fixed $y$, integrate the identity
$\operatorname{div}_x(x|x|^{-p})=(n-p)|x|^{-p}$ against $|u(x,y)|^p$ on $\{|x|>\varepsilon\}$. The inner boundary term is nonpositive. Letting $\varepsilon\downarrow0$ gives
\[
 (n-p)\int\frac{|u|^p}{|x|^p}\,d\xi
 \le p\int\frac{|u|^{p-1}}{|x|^{p-1}}|\nabla_xu|\,d\xi.
\]
The integrals are finite since $p<n$. For $p>1$, H\"older's inequality gives the result, and for $p=1$ the same estimate follows directly, using a smooth approximation to $|u|$ in the integration by parts. This is the usual Euclidean Hardy argument applied in each $x$-slice; see also~\cite{ambrosio}.
\end{proof}

\subsection{Proof of Theorem~\ref{thm_ISO_Grushin}}

By Lemma~\ref{lemma_self_improvement}, it is enough to prove $(\mathrm S_{\mathcal Q}^1)$. The proof follows the argument for \eqref{QD} manifolds, but the covering and the Hardy weight now concern only $x$.

Fix $\lambda>0$, put $q=\mathcal Q/(\mathcal Q-1)$, and let $f\in C_c^\infty(\mathbb{R}^{n+m})$. We may assume $f\ne0$ and set
\[
 R=\left(\frac{\|f\|_q}{\lambda}\right)^{q/\mathcal Q}.
\]
Apply Lemma~\ref{lemma_remoteball} to $\mathbb{R}^n$ to obtain a bounded-overlap covering $\{B_n^\alpha=B_n(x_\alpha,r_\alpha)\}_{\alpha\ge0}$, with $B_n^0=B_n(0,R)$ and $r_\alpha\sim |x_\alpha|$ for $\alpha\ge1$. Choose a subordinate smooth partition of unity $\{\mathcal X_\alpha\}$ with
\[
 0\le\mathcal X_\alpha\le1,\qquad
 |\nabla\mathcal X_\alpha|\le C r_\alpha^{-1}\quad(\alpha\ge1).
\]
Write $f_\alpha(x,y)=f(x,y)\mathcal X_\alpha(x)$. By bounded overlap,
\[
 |\{|f|\ge\lambda\}|\le C\sum_{\alpha\ge0}J_\alpha,\qquad
 J_\alpha=|\{(x,y)\in B_n^\alpha\times\mathbb{R}^m:
                    |f_\alpha(x,y)|\ge c_0^{-1}\lambda\}|,
\]
where $c_0$ is fixed independently of $R$.
For $\alpha\ge1$, decompose
$f_\alpha=(f_\alpha-e^{-tL}f_\alpha)+e^{-tL}f_\alpha$.
Lemma~\ref{le_RS} and the volume estimate imply
$\|e^{-tL}\|_{q\to\infty}\le Ct^{-\mathcal Q/(2q)}$.
Taking
\[
 t=\left(2c_0C\frac{\|f\|_q}{\lambda}\right)^{2q/\mathcal Q}
\]
therefore gives
\begin{align}\label{eq_I_alpha2}
 J_\alpha\le C\lambda^{-1}\int_0^t
       \|Le^{-sL}f_\alpha\|_{L^1(B_n^\alpha\times\mathbb{R}^m)}\,ds.
\end{align}
For $g\in L^\infty$ supported in $B_n^\alpha\times\mathbb{R}^m$, integration by parts gives
\begin{align}\label{eq_dual2}
 |\langle Le^{-sL}f_\alpha,g\rangle|
 \le\|\nabla_Lf_\alpha\|_1
       \|\nabla_Le^{-sL}g\|_{L^\infty(B_n^\alpha\times\mathbb{R}^m)}.
\end{align}
This can first be verified for smooth compactly supported $g$ and then extended by approximation and duality.

\begin{lemma}\label{ISO_Grushin_Lemma1}
For $\alpha\ge1$, $s>0$, and $g\in L^\infty$ supported in $B_n^\alpha\times\mathbb{R}^m$,
\[
 \|\nabla_Le^{-sL}g\|_{L^\infty(B_n^\alpha\times\mathbb{R}^m)}
       \le Cs^{-1/2}\|g\|_\infty,
\]
with $C$ independent of $\alpha$, $s$ and $R$.
\end{lemma}

\begin{proof}
For $x\in B_n^\alpha$ we have $|x|\sim r_\alpha$. If $\sqrt s\le C_0r_\alpha$, Lemma~\ref{le_Grushin_gradient} and the Gaussian integral estimate on a doubling space give the result directly.
Suppose now that $\sqrt s\ge C_0r_\alpha$, where $C_0$ is large enough that $\sqrt s\ge |x|$. Then
\begin{align}\label{eq_le_gradient1}
 |\nabla_Le^{-sL}g(x,y)|
 \le C\|g\|_\infty r_\alpha^{-1}s^{-\mathcal Q/2}
       \int_{B_n^\alpha}\int_{\mathbb{R}^m}
                e^{-d((x,y),(x',y'))^2/(cs)}\,dy'\,dx'.
\end{align}
For $x,x'\in B_n^\alpha$, Lemma~\ref{lemma_DS1} gives
\[
 d((x,y),(x',y'))\ge c\min\!\left(
       \frac{|y-y'|}{r_\alpha^\beta},|y-y'|^{1/(\beta+1)}\right).
\]
Using $e^{-\min(a,b)}\le e^{-a}+e^{-b}$ and integrating in polar coordinates,
\[
 \int_{\mathbb{R}^m}e^{-d((x,y),(x',y'))^2/(cs)}\,dy'
 \le C\left(s^{m/2}r_\alpha^{m\beta}
                 +s^{m(\beta+1)/2}\right)
 \le Cs^{m(\beta+1)/2}.
\]
Substitution into \eqref{eq_le_gradient1} yields
\[
 |\nabla_Le^{-sL}g(x,y)|
 \le C\|g\|_\infty r_\alpha^{n-1}s^{-n/2}
 \le Cs^{-1/2}\|g\|_\infty.
\]
The last inequality uses $n\ge1$ and $r_\alpha\le C\sqrt s$.
\end{proof}

By \eqref{eq_I_alpha2}, \eqref{eq_dual2} and Lemma~\ref{ISO_Grushin_Lemma1},
\[
 J_\alpha\le C\lambda^{-1}\sqrt t\,\|\nabla_Lf_\alpha\|_1
       =C\lambda^{-1-q/\mathcal Q}\|f\|_q^{q/\mathcal Q}
                    \|\nabla_Lf_\alpha\|_1.
\]
Bounded overlap, the partition bounds and Lemma~\ref{le_grushin_hardy} imply
\[
 \sum_{\alpha\ge1}\|\nabla_Lf_\alpha\|_1
 \le C\left(\|\nabla_Lf\|_1+\left\|\frac f{|x|}\right\|_1\right)
 \le C\|\nabla_Lf\|_1.
\]
For the central cylinder, the Hardy inequality gives
\[
 J_0\le C\lambda^{-1}\|f\|_{L^1(B_n(0,R)\times\mathbb{R}^m)}
 \le C\lambda^{-1}R\left\|\frac f{|x|}\right\|_1
 \le C\lambda^{-1-q/\mathcal Q}\|f\|_q^{q/\mathcal Q}\|\nabla_Lf\|_1.
\]
Since $1+q/\mathcal Q=q$, we conclude that
\[
 \sup_{\lambda>0}\lambda|\{|f|\ge\lambda\}|^{1/q}
 \le C\|\nabla_Lf\|_1^{1/q}\|f\|_q^{1/\mathcal Q}.
\]
Finally, $\|f\|_q\le\|f\|_\infty|\operatorname{supp}f|^{1/q}$, so Proposition~\ref{le_BCLS} applies with $r_0=s_0=q$, $\theta_0=1/q$ and $W(f)=\|\nabla_Lf\|_1$. It gives $(\mathrm S_{\mathcal Q}^1)$, and Lemma~\ref{lemma_self_improvement} gives the remaining exponents. No integer restriction on $\beta$ is used.

\appendix
\section{\texorpdfstring{$\Gamma$}{Gamma}-calculus}\label{appendix}

For completeness, we verify \eqref{CD1} with the sign convention used throughout the paper: $L\ge0$ and $\mathcal L=-L\le0$. All calculations take place on $\{x\ne0\}$, where the coefficients are smooth. Let $a(x)=|x|^{2\beta}$ and use indices $1\le i,k\le n$, $1\le j,l\le m$. Then
\[
 \Gamma(f)=\sum_i f_{x_i}^2+a\sum_j f_{y_j}^2,
 \qquad
 \mathcal Lf=\sum_i f_{x_i x_i}+a\sum_jf_{y_j y_j}.
\]
The product rule for $\mathcal L$ is
$\mathcal L(vw)=v\mathcal Lw+w\mathcal Lv+2\Gamma(v,w)$.
Consequently,
\begin{align}\label{appendix1}
 \tfrac12\mathcal L\Gamma(f)
 ={}&S_{xx}+2aS_{xy}+a^2S_{yy}
       +\sum_i f_{x_i}\mathcal L(f_{x_i})
       +a\sum_j f_{y_j}\mathcal L(f_{y_j})\nonumber\\
 &+\tfrac12\Bigl(\sum_i\partial_{x_i}^2a\Bigr)|\nabla_yf|^2
       +2\sum_{i,j}(\partial_{x_i}a)f_{y_j}f_{x_i y_j}.
\end{align}
On the other hand,
\[
 (\mathcal Lf)_{x_i}=\mathcal L(f_{x_i})
                         +(\partial_{x_i}a)\sum_jf_{y_j y_j},
 \qquad
 (\mathcal Lf)_{y_j}=\mathcal L(f_{y_j}),
\]
so that
\[
 \Gamma(f,\mathcal Lf)
 =\sum_i f_{x_i}\mathcal L(f_{x_i})
       +a\sum_jf_{y_j}\mathcal L(f_{y_j})
       +\sum_i(\partial_{x_i}a)f_{x_i}\sum_jf_{y_j y_j}.
\]
Subtracting this identity from \eqref{appendix1}, and using
\[
 \partial_{x_i}a=2\beta|x|^{2\beta-2}x_i,
 \qquad
 \sum_i\partial_{x_i}^2a=2\beta(n+2\beta-2)|x|^{2\beta-2},
\]
gives
\begin{align}\label{appendix2}
 \Gamma_2(f)={}&S_{xx}+2|x|^{2\beta}S_{xy}+|x|^{4\beta}S_{yy}
   +\beta(n+2\beta-2)|x|^{2\beta-2}|\nabla_yf|^2\nonumber\\
 &+4\beta|x|^{2\beta-2}\sum_{i,j}x_if_{y_j}f_{x_i y_j}
   +2\beta|x|^{2\beta-2}\Delta_yf\sum_i x_if_{x_i}.
\end{align}
In the last term $\Delta_y=-\sum_j\partial_{y_j}^2$, explaining its positive sign.

\bigskip
{\bf Acknowledgments.} 
Part of this article is contained in the author's Ph.D. thesis; see the thesis abstract \cite{He_phdthesis}. He would like to thank his supervisor, Adam Sikora, for his patient guidance and consistent encouragement. DH was supported by the China Postdoctoral Science Foundation No.~2026T190831 and the NNSF of China No.~42381947.

{\bf Data availability.}
This work has neither used nor created new data.

{\bf Conflict of interest.}
The author declares that there is no conflict of interest.

\bibliographystyle{abbrv}

\bibliography{Iso_references}

\end{document}